\documentclass[]{amsart}

\usepackage{amsmath, amsfonts, amsthm, amssymb}
\usepackage[utf8]{inputenc} 
\usepackage[T1]{fontenc}   
\usepackage{graphicx}
\usepackage{float}

\usepackage{geometry}
\newcounter{relctr} %% <- counter for relations
\everydisplay\expandafter{\the\everydisplay\setcounter{relctr}{0}} %% <- reset every eq
\AtBeginDocument{} %% <- store original definition

\usepackage[table]{xcolor}
\usepackage{stmaryrd}
\usepackage{bm}
\usepackage{bbm}
\usepackage{tikz}
\usepackage{tikz-cd}
\usepackage[shortlabels]{enumitem}
\setlist[1]{wide}
\setlist[2]{leftmargin=15mm}
\setlist[enumerate]{label=\rm{(\roman*)}}
\setlist[itemize]{label=\raisebox{0.25ex}{\tiny$\bullet$}}

\usepackage[backref, colorlinks, linktocpage, citecolor = blue, linkcolor = blue]{hyperref}

\usepackage{soul}
\usepackage{mathtools}

\usepackage{lmodern}
\usepackage[noadjust]{cite}
\usepackage{framed}
\usepackage{stackrel}
\usepackage{array}
\usepackage{xifthen}

\usepackage[english]{babel}
\usepackage{csquotes}

\DeclareQuoteStyle[american]{english}
{\itshape\textquotedblleft}
[\textquotedblleft]
{\textquotedblright}
[0.05em]
{\textquoteleft}
{\textquoteright}

\usepackage{booktabs}
\usepackage{longtable}

\usepackage[all, cmtip]{xy}

\makeatletter
\@namedef{subjclassname@2020}{\textup{2020} Mathematics Subject Classification}
\makeatother

\numberwithin{equation}{section}

\newtheorem{theorem}{Theorem}[section]
\newtheorem{proposition}[theorem]{Proposition}
\newtheorem{lemma}[theorem]{Lemma}

\theoremstyle{definition}
\newtheorem{remark}[theorem]{Remark}

\newtheorem{definition}[theorem]{Definition}

\newtheoremstyle{customNumber}
     {}          % Space above (empty = default)
     {}          % Space below
     {\itshape}  % Body font
     {}          % Indent amount (empty = no indent)
     {\bfseries} % Thm head font
     {.}         % Punctuation after thm head
     { }         % Space after thm head (\newline = linebreak)
     {\thmname{#1}\thmnumber{ #2}\thmnote{ #3}}
\theoremstyle{customNumber}

\renewcommand{\phi}{\varphi}

\newcommand\C{{\mathbb C}}
\newcommand\GL{{\mathrm{GL}}}

\newcommand\Pic{{\mathrm{Pic}}}
\newcommand\PP{{\mathbb P}}
\newcommand\R{{\mathbb R}}
\newcommand\A{{\mathbb A}}
\newcommand\Z{{\mathbb Z}}

\newcommand{\F}{\mathbb{F}}

\newcommand{\Aut}{\mathrm{Aut}}
\newcommand{\Bir}{\mathrm{Bir}}

\newcommand{\id}{\mathrm{id}}

\newcommand{\Base}{\mathrm{Base}}
\newcommand{\charpoly}{\mathrm{char}}

\title[The ordinal of dynamical degrees of birational maps]
{The ordinal of dynamical degrees \\ of birational maps of the projective plane}

\author{Anna Bot}
\address{}

\keywords{Dynamical degree, rational projective surfaces, ordinals \\ The author acknowledges support by the Swiss National Science Foundation Grant ``Geometrically ruled surfaces'' 200020-192217.}

\begin{document}

\begin{abstract}
	We show that the ordinal of the dynamical degrees of all complex birational maps of the projective plane is $\omega^{\omega}$.
\end{abstract}

\thanks{}
\maketitle

%\tableofcontents

\section{Introduction} \label{section: introduction}

Fix a smooth projective surface $X$ over a field $\mathbf{k}$. Any birational self-map $f$ of $X$ can be analysed by considering the dynamical system its iterates provide. A way to measure how chaotic such a system might become is by considering the dynamical degree of $f\colon X\dashrightarrow X$, which is given by
\begin{align*}
	\lambda(f)=\lim_{n \rightarrow \infty} (D\cdot(f^n)_\ast D)^{1/n},
\end{align*}
where $D \subset X$ is any ample divisor and $D\cdot C$ denotes the intersection form. If $X$ is the projective plane, the definition agrees with $\lambda(f) = \lim_{n \rightarrow \infty} (\deg(f^n))^{1/n}$. 

The dynamical degree has been considered in many different contexts, often due to its close connection with entropy, see for example  \cite{gromov1987entropy}, \cite{yomdin1987volume}, \cite{MR1704282} \cite{cantat2001dynamique}, \cite{MR1867314}, \cite{gromov2003entropy}, \cite{bedford2006periodicities}, \cite{MR2354205}, \cite{bedford2009dynamics}, \cite{bedford2010continuous}, \cite{MR2825269}, \cite{deserti2011automorphisms}, \cite{junyi2015periodic}, \cite{mcmullen2016automorphisms}, \cite{MR3454379}, \cite{MR3667901}, \cite{jonsson2018complex}, \cite{MR4292865}. By considering all dynamical degrees of all birational maps of $X$, we can gauge how differently dynamical systems on $X$ might behave. Due to any dynamical degree on a rational projective surface being an algebraic integer $\geq 1$ (see for example the results of Diller and Favre~\cite[Theorem~$5.1$]{MR1867314}), the set $\Lambda(\Bir_{\mathbf{k}}(X)) \subset \R$ of all dynamical degrees possible for elements of $\Bir_{\mathbf{k}}(X)$ is at most countable. Yet, this is a quantitative statement, since it takes into consideration only the cardinality of $\Lambda(\Bir_{\mathbf{k}}(X))$; using the theory of ordinal numbers instead, we are able to make a qualitative statement about accumulation points of $\Lambda(\Bir_{\mathbf{k}}(\PP^2_{\mathbf{k}}))$.

Recall that a well ordered set is a totally ordered set such that every non-empty subset contains a smallest element. Any such well ordered set is up to order-preserving bijection equal to an ordinal. The ordinals of finite cardinality are in bijection with the natural numbers, and we call $\omega$ the first infinite ordinal. Using ordinals gives us the possibility of putting elements of $\Lambda(\Bir_{\mathbf{k}}(X))$ into relation with one another and describe different flavours of infinity that the sets $\Lambda(\Bir_{\mathbf{k}}(X))$ of the same cardinality can have.

Using an argument by Xie~\cite{junyi2015periodic}, Blanc and Cantat showed in~\cite[Theorem~$7.2$]{MR3454379} that the set of all dynamical degrees of all projective surfaces is well ordered with respect to the standard order on $\R$. Yet, by Blanc and Cantat~\cite[Theorem~B]{MR3454379}, the union of all $\Lambda(\Bir_{\mathbf{k}}(X))$ where $X$ is geometrically non-rational and where $\mathbf{k}$ is arbitrary is a discrete closed subset of $\R$ and thus the ordinal is at most $\omega$. So the question boils down to: for a given geometrically rational projective surface $X$, what is the ordinal of $\Lambda(\Bir_{\mathbf{k}}(X))$? We show that it is $\omega^\omega$ if $\mathbf{k}=\C$. More generally, we have: 

\begin{theorem}\label{thm: main theorem}
	For any field $\mathbf{k}$ and any projective geometrically rational surface $X$ defined over $\mathbf{k}$, the order type of $\Lambda(\Bir_{\mathbf{k}}(X)) \subset \R$ is bounded above by $\omega^\omega$. If, in addition, $\mathbf{k}$ contains the real algebraic numbers and $X$ is rational over $\mathbf{k}$, then $\Lambda(\Bir_{\mathbf{k}}(X))=\omega^\omega$.
\end{theorem}
For any rational surface $X$, the set $\Lambda(\Bir_{\mathbf{k}}(X))$ agrees with $\Lambda(\Bir_{\mathbf{k}}(\PP^2_{\mathbf{k}}))$; we may thus restrict all statements proved in this article to $\PP^2_{\mathbf{k}}$. One can interpret the second statement of Theorem~\ref{thm: main theorem} as follows: there exists a sequence of birational transformations whose dynamical degrees each are limits of a sequence of dynamical degrees of other birational transformations. Now, each of these dynamical degrees is again the limit of a sequence, whose members are again limits of sequences, and so on. In a sense, this means that $\Lambda(\Bir_{\mathbf{k}}(\PP^2_{\mathbf{k}}))$ is a very dense countable set whenever $\mathbf{k}$ contains the real algebraic numbers and $X$ is rational over $\mathbf{k}$.

One can prove the lower bound under these assumptions in two ways, either using the Weyl group as in Section~\ref{section: bounding from below, Weyl group}, or by explicitly constructing rational surface automorphisms with the desired dynamical degrees as in Section~\ref{section: bounding from below, realisations}. The first strategy was pointed out to the author by Curtis T McMullen: recall that one can define the Weyl group~$W_n=W(E_n)$ the Coxeter group determined by the Coxeter-Dynkin diagram $E_n$, for which Nagata~\cite{MR126444} proves that the image of $\Aut(X) \rightarrow \Aut(\Pic(X))$ lies in $W_n$. Applying Theorem~$1.1$ of Uehara~\cite{MR3477879} shows that the Weyl spectrum~$\Lambda(W)$ consisting of all spectral radii of all elements of all $W_n$ with $n\geq 3$ is contained in $\Lambda(\Bir(\PP^2_{\C}))$. It is thus well ordered, but it is not equal to $\Lambda(\Bir(\PP^2_{\C}))$, since $\Lambda(W)$ is comprised of Salem numbers and of reciprocal quadratic integers, whereas $\Lambda(\Bir(\PP^2_{\C})) \setminus \Lambda(W)$ are Pisot numbers (see~\cite[$1.1.2$]{MR3454379} for an introduction to Salem and Pisot numbers, or Remark~\ref{rem: last remark}). Using Section~\ref{section: arithmetic properties of auxiliary polynomials}, we can prove:
\begin{theorem}\label{thm: main theorem for Weyl spectrum}
	The Weyl spectrum $\Lambda(W)$ has order type~$\omega^\omega$.
\end{theorem}
We can prove the above theorem without ever having to construct explicit birational maps (see Theorem~\ref{thm: main theorem, but with weyl group}). From this, the lower bound on $\Lambda(\Bir(\PP^2_{\C}))$ follows directly. Yet, as the proof of Uehara's strong result is relatively long, we also demonstrate the self-contained second version of the proof in Section~\ref{section: bounding from below, realisations}; in addition, Section~\ref{section: bounding from below, realisations} provides explicit realisations of de Jonqui\`eres maps.

For both strategies, we construct, for a fixed degree~$d$, suitable sequences of elements of the Weyl group which in the second strategy are realised as birational maps, of degree~$d$ whose spectral radii, which after the realisation equal the dynamical degrees, form a set of ordinal $\omega^{2d-2}$. Yet, in the open interval topology on $\R$, even though the set has ordinal $\omega^{2d-2}$, it does not contain any accumulation point. All the dynamical degrees we construct belong to a set $\Lambda_{d, 2d-1}$ consisting of specific birational maps which can be realised as automorphisms on some blow-up of $\PP^2_{\mathbf{k}}$, and are thus Salem numbers, yet their accumulation points are Pisot numbers. However, the accumulation points also belong to $\Lambda(\Bir_{\mathbf{k}}(\PP^2_{\mathbf{k}}))$ whenever $\mathbf{k}$ is uncountable and algebraically closed, and thus for $\mathbf{k}=\C$, since by Blanc and Cantat~\cite[Theorem~D]{MR3454379}, the set $\Lambda(\Bir_{\mathbf{k}}(\PP^2_{\mathbf{k}}))$ is closed under this assumption. Moreover, our construction does indeed produce these accumulation points: in fact, the closure of $\Lambda_{d, 2d-1}$ is equal to $\bigsqcup_{1 \leq m \leq 2d-1}\Lambda_{d, m}$, where the accumulation points of $\Lambda_{d, m}$ are precisely the elements of $\Lambda_{d, m-1}$. The set $\Lambda_{d, 1}$ comprises only  $\tfrac{1}{2}(d-1+\sqrt{d^2-2d+5})$, which is relatively easy to determine; however, already determining the dynamical degrees in the sequence constituting $\Lambda_{d, 2}$ converging to this value is impractical, since they are given as the largest roots of polynomials of increasing degrees. Nonetheless, the polynomials are given explicitly in this text.

This idea of considering the orbits of these base points has been developed in McMullen~\cite{MR2354205}, or in Diller~\cite{MR2825269} for birational maps of degree $2$. These orbits are then the independent parameters providing us with the sequences we are looking for. In fact, Diller~\cite{MR2825269} gives a complete description of the quadratic birational transformations whose base points lie on an irreducible cubic curve, so one could use that to give a complete description of all dynamical degrees of degree~$2$ maps. This stands in contrast to our result: we do construct suitable sequences of dynamical degrees, but there may be many more dynamical degrees which we do not catch in this way.

The article is structured in the following way: in Section~\ref{section: bounding from above}, the upper bound $\Lambda(\Bir_{\mathbf{k}}(\PP^2_{\mathbf{k}}))\leq \omega^\omega$ is shown. Then, in Section~\ref{section: arithmetic properties of auxiliary polynomials}, some arithmetic facts and tools are collected, which are then used in Sections~\ref{section: bounding from below, Weyl group} and~\ref{section: bounding from below, realisations} to give the bound $\omega^\omega \leq \Lambda(\Bir_{\mathbf{k}}(\PP^2_{\mathbf{k}}))$ for fields containing the real algebraic integers. What is somewhat hidden when phrasing the result using ordinals is that when given a suitable matrix $F$ in $\GL_{n+1}(\Z)$ and an eigenvalue $a\in \C \setminus\{1\}$, we give in Section~\ref{section: bounding from below, realisations}, under some further conditions, explicit points on a cuspidal cubic depending on $a$ such that their blow-up admits a birational map realising $F$, see Proposition~\ref{prop: eigenvector implies existence on cuspidal cubic}. The number $a$ being a real algebraic integer, we need our base field $\mathbf{k}$ to contain the real algebraic numbers.

For the rest of the paper, we omit the field $\mathbf{k}$ from the notation, and write $\Lambda(\Bir(X))$ instead of $\Lambda(\Bir_{\mathbf{k}}(X))$, and $\PP^2$ instead of $\PP^2_{\mathbf{k}}$.
~\\

\textbf{Acknowledgements.} I would like to thank Serge Cantat, St\'ephane Lamy and Konstantin Loginov for clarifying discussions, mostly, but not exclusively, in relation to ordinals. Furthermore, I am grateful to Curtis T McMullen for drawing my attention to how Uehara's article can shorten the proof of Theorem~\ref{thm: main theorem}, giving rise to Section~\ref{section: bounding from below, Weyl group}. Also, thank you to Gabriel Dill and again Stéphane Lamy for suggesting improvements on the first draft. As always, I thank my PhD advisor J\'er\'emy Blanc for suggesting the question and for the many helpful discussions.

\section{Bounding from above} \label{section: bounding from above}

Denote by $\Bir(\PP^2)$ the set of birational maps $f:\PP^2 \dashrightarrow \PP^2$, and by $\Bir_d(\PP^2)$ all birational maps of $\PP^2$ of degree~$d$, meaning that the map $f$ can be described by three homogeneous polynomials of degree~$d$ having no common divisor. Thus, $\Bir_d(\PP^2)$ is a subset of some projective space. One can see that it is locally closed, and therefore, $\Bir_d(\PP^2)$ can be seen as an algebraic variety and endowed with a Zariski topology, see also~\cite[Proposition~$2.15$]{MR3092478}. For any subset $S \subset \Bir(\PP^2)$, write $\Lambda(S)$ for the set of all dynamical degrees that elements of $S$ can attain. As the degree is greater or equal than $1$, the set~$\Lambda(S)$ is contained in $\R^{\geq 1}$.

Let $X$ be a topological space and $I \subset \R \cup \{-\infty, \infty\}$ a subset. A function $g \colon X \rightarrow I$ is called \emph{lower semicontinuous} if for any $r \in I$, the set $L(r) = \left\{ x \in X \mid g(x) \leq r \right\}$ is closed in $X$. The following result of Xie~\cite{junyi2015periodic} shows that the dynamical degree as a function $\lambda \colon \Bir_d(\PP^2) \rightarrow \R^{\geq 1}$ is lower-semicontinuous. 
\begin{theorem}[{\cite{junyi2015periodic}, Theorem~$1.5$}] \label{thm: junyi}
	Let $\mathbf{k}$ be an algebraically closed field and $d \geq 2$ be an integer. Then for any $\lambda < d$, the set $U_{\lambda}=\{f \in \Bir_d(\PP^2) \thinspace | \thinspace \lambda(f) > \lambda \}$ is a Zariski dense open set of $\Bir_d(\PP^2)$.
\end{theorem}
A totally ordered set is called \emph{well ordered} if every subset has a minimum. It is a nice exercise to prove that, assuming the axiom of choice, well orderedness is equivalent to any decreasing sequence becoming stationary; it can also be found in Chapter~\MakeUppercase{\romannumeral 3},~§~$6.5$, Proposition~$6$ of \cite{MR2102219}. Recall that a well ordered set corresponds to a unique ordinal, that is its order type, and that we have a total order on ordinals. As usual, we denote by $\omega$ the smallest infinite ordinal, that is the order type of the natural numbers with the standard order.

The next proposition shows that $\Lambda(\Bir_d(\PP^2))$ is a well ordered set, a fact readily implied by $\Lambda(\Bir(\PP^2))$ being well ordered, which is proven in \cite{MR3454379}. Then, since $\Lambda(\Bir_d(\PP^2))$ is well ordered, we can describe its order type.

\begin{proposition} \label{prop: order type for general variety with morphism to Bir}
	Let $\mathbf{k}$ be an arbitrary field, $d \geq 1$ an integer and $X \subset \Bir_d(\PP^2)$ a locally closed subset. Then, $\Lambda(X)$ is a  well ordered subset of $\R$, whose ordinal is smaller or equal to $\omega^{\dim(X)}n+n$, where $n$ is the number of irreducible components of $X$.
\end{proposition}
\begin{proof}
	Suppose that $\mathbf{k}$ is algebraically closed; if the result holds in that case, then for any field $\mathbf{k}$, we can use that any birational map defined over $\mathbf{k}$ is also defined over $\overline{\mathbf{k}}$, and deduce the claim. The set $\Lambda(\Bir(\PP^2))$ of all dynamical degrees of birational maps of $\PP^2$ is well ordered by \cite[Theorem $7.2$]{MR3454379}, and it contains $\Lambda(X)$. Thus, $\Lambda(X)$ must be well ordered, too. 
	
	We fix $d \geq 1$ and prove the second claim by induction on the dimension of $X$. If $\dim(X)=0$, then any irreducible component of $X$ is a point, and the claim follows. Suppose the claim is proved for any locally closed subvariety of $\Bir_d(\PP^2)$ of dimension $c-1 \geq 0$, and consider $X \subset \Bir_d(\PP^2)$ locally closed with $\dim(X)=c$. Decompose $X$ into its irreducible components $X_1\cup \ldots \cup X_n$ and pick $X_i$ with $\dim(X_i)=\dim(X)=c$; for the irreducible components of dimension smaller than $c$, the ordinal is smaller or equal to $\omega^{\dim(X)}+1$ by induction.
	
	Assume by contradiction that the order type of $\Lambda(X_i)$ is strictly greater than $\omega^{c}+1$. Then there exist $\nu_1 < \nu_2$ in $\Lambda(X_i)$ such that $\left\{ \nu \in \Lambda(X_i) \mid \nu <\nu_1 \right\}$ is of order type $\omega^c$. By lower semicontinuity proven in \cite[Theorem $1.5$]{junyi2015periodic} (see Theorem~\ref{thm: junyi}), the set $L(\nu_1)=\{ x \in \Bir_d(\PP^2) \mid \lambda(x) \leq \nu_1 \}$ is closed in $\Bir_d(\PP^2)$, and thus $L(\nu_1)\cap X_i$ is closed in $X_i$. If the dimension of this set were strictly smaller than $c$, its image under the map $\lambda\colon \Bir_d(\PP^2)\rightarrow\R^{\geq 1}$ in $\Lambda(X_i)$ would by induction be smaller or equal to $\omega^{c-1}m+m$ for some $m$ counting the irreducible components of $L(\nu_1)\cap X_i$. But $\omega^{c-1}m+m<\omega^c$, a contradiction. Thus, $L(\nu_1)\cap X_i=X_i$, which in turn implies that $\nu_2 \leq \nu_1$, a contradiction. Therefore, any irreducible component $X_i$ of $X$ satisfies $\Lambda(X_i)\leq \omega^c+1$ and hence $\Lambda(X)\leq \omega^c n +n$. This finishes the induction an proves the proposition.
\end{proof}

\begin{theorem} \label{thm: bounding from above}
	Over any field $\mathbf{k}$, the ordinal of $\Lambda(\Bir(\PP^2))$ is less or equal to $\omega^{\omega}$.
\end{theorem}
\begin{proof}
	By Proposition~\ref{prop: order type for general variety with morphism to Bir}, we see that $\Lambda(\Bir_d(\PP^2))$ is of order type strictly less than $\omega^{\dim(\Bir_d(\PP^2))+1}=\omega^{4d+7}$, where $\dim(\Bir_d(\PP^2))=4d+6$ by \cite[Theorem~$1$]{MR3319964}. As $\Bir(\PP^2)=\underset{d\geq 1}{\bigcup} \Bir_d(\PP^2)$, we find that $\Lambda(\Bir(\PP^2)) =  \underset{d\geq 1}{\bigcup}\Lambda(\Bir_d(\PP^2))$ is of order type at most~$\omega^\omega$.
\end{proof}

\section{Arithmetic properties of auxiliary polynomials} \label{section: arithmetic properties of auxiliary polynomials}

To bound $\Lambda(\Bir(\PP^2))$ from below, we first collect arithmetic facts about very specific, auxiliary polynomials and their largest real roots, all of which plays a part in Sections~\ref{section: bounding from below, Weyl group} and~\ref{section: bounding from below, realisations}. This current section thus may be skipped if one would first like to understand why the polynomials below and their largest positive roots are of importance, and reference back to the results once they are used. The polynomials of interest are the following:

\begin{definition} \label{def: auxiliary polynomials}
	Fix $d \geq 1$, $1 \leq m \leq 2d-1$ and $\underline{n}=(n_2, \ldots, n_m)$. Abbreviate $I_m=\{2, \ldots, m\}$. Then, define \[p_{d, \underline{n}}(X) = (X^2-(d-1)X-1)\prod_{i=2}^{m}(X^{n_i}+1)+X\sum_{i=2}^{m}\prod_{j \in I_m\setminus \{i\}}(X^{n_j}+1).\]
\end{definition}

The polynomial $p_{d,\underline{n}}(X)$ appears as a factor in the characteristic polynomial of a certain matrix: fix $d\geq 1$, $1 \leq m \leq 2d-1$ and $\underline{n}=(n_2, \ldots, n_{m})$ and consider the matrix
\begin{align}
	J_d^{\underline{n}}=\left(\begin{array}{c|c|c|c|c|c}
		d & \begin{smallmatrix}
			0 & ~  & \phantom{-}d-1
		\end{smallmatrix} & \begin{smallmatrix}
			0 & \,\scalebox{1.005}{$\cdots$} & 0 & \phantom{-}1
		\end{smallmatrix} & \begin{smallmatrix}
			0 & \,\scalebox{1.005}{$\cdots$} & 0 & \phantom{-}1
		\end{smallmatrix}& \cdots &  \begin{smallmatrix}
			0 & ~& \,\scalebox{1.005}{$\cdots$} & ~& 0 & \phantom{-}1
		\end{smallmatrix} \\
		\hline
		\begin{smallmatrix}
			~\\~\\-(d-1) \\ ~\\ 0 \\ ~
		\end{smallmatrix} & \begin{smallmatrix}
			~ & ~ \\ ~ & ~ \\ 0 & -(d-2) \\ ~ & ~ \\ 1 & \phantom{-}0 \\~ & ~ 
		\end{smallmatrix}& \begin{smallmatrix}
			~\\
			0 & \,\scalebox{1.005}{$\cdots$} & 0 &-1 \\ ~\\
			0 & \,\scalebox{1.005}{$\cdots$} & 0& \phantom{-}0
		\end{smallmatrix} & \begin{smallmatrix}
			~\\
			0 & \,\scalebox{1.005}{$\cdots$} & 0 &-1 \\ ~\\
			0 & \,\scalebox{1.005}{$\cdots$} & 0& \phantom{-}0
		\end{smallmatrix} & \cdots & \begin{smallmatrix}
			~\\
			0 & ~& \,\scalebox{1.005}{$\cdots$} & ~& 0 &-1 \\ ~\\
			0 & ~&\,\scalebox{1.005}{$\cdots$} & ~& 0& \phantom{-}0
		\end{smallmatrix} \\
		\hline
		\begin{smallmatrix}
			~\\~\\-1 \\ ~ \\ 0 \\ \vdots\\ ~ \\ 0 \\ ~
		\end{smallmatrix} & \begin{smallmatrix}
			~\\ 0 & ~ & -1 & ~ \\ ~\\ 0 &~ & \phantom{-}0& ~ \\ \vdots &~ & \phantom{-}\vdots& ~ \\ ~\\ 0 &~ & \phantom{-}0& ~
		\end{smallmatrix} & \begin{smallmatrix}
			0 & \scalebox{1.005}{$\cdots$} & 0 & -1 \\ ~\\~& ~& ~ & \phantom{-}0 \\ ~ & \raisebox{0.15em}{$\mathbbm{1}_{n_2-1}$} & ~ & \phantom{-}\vdots \\~\\~& ~& ~ & \phantom{-}0
		\end{smallmatrix}  & \bf{0}&\cdots & \bf{0}  \\
		\hline
		\begin{smallmatrix}
			~\\~\\-1 \\ ~ \\ 0 \\ \vdots\\ ~ \\ 0 \\ ~
		\end{smallmatrix} & \begin{smallmatrix}
			~\\ 0 & ~ & -1 & ~ \\ ~\\ 0 &~ & \phantom{-}0& ~ \\ \vdots &~ & \phantom{-}\vdots& ~ \\ ~\\ 0 &~ & \phantom{-}0& ~
		\end{smallmatrix}  & \bf{0} & \begin{smallmatrix}
			0 & \scalebox{1.005}{$\cdots$} & 0 & -1 \\ ~\\~& ~& ~ & \phantom{-}0 \\ ~ & \raisebox{0.15em}{$\mathbbm{1}_{n_3-1}$} & ~ & \phantom{-}\vdots \\~\\~& ~& ~ & \phantom{-}0
		\end{smallmatrix}  & \cdots & \bf{0}  \\
		\hline
		\vdots& \vdots & \vdots & \vdots  & \ddots & \vdots  \\
		\hline
		\begin{smallmatrix}
			~\\~\\-1 \\ ~ \\ 0 \\ \vdots\\ ~ \\ 0 \\ ~
		\end{smallmatrix} & \begin{smallmatrix}
			~\\ 0 & ~ & -1 & ~ \\ ~\\ 0 &~ & \phantom{-}0& ~ \\ \vdots &~ & \phantom{-}\vdots& ~ \\ ~\\ 0 &~ & \phantom{-}0& ~
		\end{smallmatrix}  & \bf{0} & \bf{0}  & \cdots & \begin{smallmatrix}
			0 & \scalebox{1.005}{$\cdots$} & 0 & -1 \\ ~\\~& ~& ~ & \phantom{-}0 \\ ~ & \raisebox{0.15em}{$\mathbbm{1}_{n_{m}-1}$} & ~ & \phantom{-}\vdots \\~\\~& ~& ~ & \phantom{-}0
		\end{smallmatrix}
	\end{array}\right), \label{eq: block matrix}\tag{$\diamondsuit$}
\end{align}
with $\mathbbm{1}_{n_k-1}$ the identity matrix of dimension $n_k-1$. Here, $\bf{0}$ is short hand for the zero matrix, always of the respective suitable dimension.

\begin{lemma} \label{lemma: characteristic poly of all de Jonquieres}
	Consider $d\geq 1$, $1 \leq m \leq 2d-1$ and $\underline{n}=(n_2, \ldots, n_{m})$. Then, the characteristic polynomial $\charpoly_{d,\underline{n}}(X)$ of $J_d^{\underline{n}}$ satisfies $\charpoly_{d,\underline{n}}(X)=(X-1)p_{d, \underline{n}}(X)$.
\end{lemma}
\begin{proof}
	To calculate the characteristic polynomial of $J^{\underline{n}}_d$ consider first the last $n_{m}+1$ rows of $X\mathbbm{1}-J^{\underline{n}}_d$:
	\begin{align*}
		\left(\begin{array}{c|c|c|c|c|c}
			\begin{smallmatrix}
				1 \\ ~ \\ 0 \\ ~ \\ \vdots \\ ~ \\ ~ \\ 0 
			\end{smallmatrix} & \begin{smallmatrix}
				0 & ~ & 1 & ~ \\  ~\\ 0 &~ & 0& ~ \\ ~ \\ \vdots &~ & \vdots& ~  \\ ~ \\ ~\\ 0 &~ & 0& ~
			\end{smallmatrix}  & \bf{0} &  \cdots & \bf{0}  & \begin{smallmatrix}
				X & 0 & \cdots & 0 & 1 \\ 
				~ \\
				-1 & X & ~ & ~ & 0 \\ 
				~ & -1& \ddots & ~ & \vdots \\
				~ & ~ & \ddots & X & 0 \\ ~ \\
				~ & ~ & ~ & -1 & X
			\end{smallmatrix}
		\end{array}\right).
	\end{align*}
	With row manipulations, this can be reduced to 
	\begin{align*}
		\left(\begin{array}{c|c|c|c|c|c}
			\begin{smallmatrix}
				1 \\ ~ \\ 0 \\ ~ \\ \vdots \\ ~ \\ ~ \\ 0 
			\end{smallmatrix} & \begin{smallmatrix}
				0 & ~ & 1 & ~ \\  ~\\ 0 &~ & 0& ~ \\ ~ \\ \vdots &~ & \vdots& ~  \\ ~ \\ ~\\ 0 &~ & 0& ~
			\end{smallmatrix}  & \bf{0} &  \cdots & \bf{0}  & \begin{smallmatrix}
				\phantom{-}0 & ~ & \phantom{-}0 & \cdots & \phantom{-}0 & \phantom{-}X^{n_{m}}+1 \\ 
				~ \\
				-1 & ~& \phantom{-}0 & ~ & ~ & \phantom{-}X^{n_{m}-1} \\ 
				~ & ~& -1 & \ddots & ~ & \vdots \\
				~ & ~& ~ & \ddots& \phantom{-}0 & \phantom{-}X^2 \\ ~ \\
				~ & ~& ~ & ~ & -1 & X
			\end{smallmatrix}
		\end{array}\right).
	\end{align*}
	Repeating this for the other rows and using Laplace expansion on the columns with the $-1$ entries, we find
	\begin{align}
		\charpoly_{d,\underline{n}}(X)&=\det \left( \begin{array}{c|c}
			\begin{smallmatrix}
				X-d & \phantom{-}0 & -(d-1) \\ d-1 & \phantom{-}X & \phantom{-}d-2 \\ 0 & -1 & \phantom{-}X \\ ~
			\end{smallmatrix} & \begin{smallmatrix}
				-1 & ~ &~ &~ & -1 &~ &~ &  \scalebox{1.005}{$\cdots$}~ & & -1 \\ \phantom{-}1 & ~&~ &~ &\phantom{-}1 &~ &~ &  \scalebox{1.005}{$\cdots$} &~ & \phantom{-}1 \\ \phantom{-}0 &~ &~ & ~ &  \phantom{-}0 &~ &~ &  \scalebox{1.005}{$\cdots$} &~ & \phantom{-}0 \\ ~
			\end{smallmatrix} \\ \hline
			\begin{smallmatrix}
				~ \\ ~\\
				1 & ~ &~ & 0 &~ &~ & \phantom{-}1 \\ ~ \\1 &~ &~ & 0 &~ &~ & \phantom{-}1 \\ \vdots &~ &~ & \vdots &~ &~ & \phantom{-}\vdots\\ ~ \\ 1 &~ &~ & 0 &~ &~ & \phantom{-}1
			\end{smallmatrix} & \begin{smallmatrix}
				~ \\ ~\\
				\phantom{--}X^{n_2}+1 & ~ & ~ & ~ \\~\\ ~ & \phantom{-}X^{n_3}+1 &~ & ~ \\ ~ & ~ &\ddots & ~ \\ ~\\~ & ~ & ~ & X^{n_{m}}+1
			\end{smallmatrix}
		\end{array}\right) \nonumber \\
		&=\det \left( \begin{array}{c|c}
			\begin{smallmatrix}
				X-d & \phantom{-}0 & -(d-1) \\ 
				X-1 & \phantom{-}X & -1 \\ 
				0 & -1 & \phantom{-}X \\ ~
			\end{smallmatrix} & \begin{smallmatrix}
				-1 & ~ &~ &~ & -1 &~ &~ &  \scalebox{1.005}{$\cdots$}~ & & -1 \\ \phantom{-}0 & ~&~ &~ &\phantom{-}0 &~ &~ &  \scalebox{1.005}{$\cdots$} &~ & \phantom{-}0 \\ \phantom{-}0 &~ &~ & ~ &  \phantom{-}0 &~ &~ &  \scalebox{1.005}{$\cdots$} &~ & \phantom{-}0 \\ ~
			\end{smallmatrix} \\ \hline
			\begin{smallmatrix}
				~ \\ ~\\
				1 & ~ &~ & 0 &~ &~ & \phantom{-}1 \\ ~ \\1 &~ &~ & 0 &~ &~ & \phantom{-}1 \\ \vdots &~ &~ & \vdots &~ &~ & \phantom{-}\vdots\\ ~ \\ 1 &~ &~ & 0 &~ &~ & \phantom{-}1
			\end{smallmatrix} & \begin{smallmatrix}
				~ \\ ~\\
				\phantom{--}X^{n_2}+1 & ~ & ~ & ~ \\~\\ ~ & \phantom{-}X^{n_3}+1 &~ & ~ \\ ~ & ~ &\ddots & ~ \\ ~\\~ & ~ & ~ & X^{n_{m}}+1
			\end{smallmatrix}
		\end{array}\right) \label{eq: laplace}
	\end{align}
	If we apply Laplace expansion to the fourth column, we find that $\charpoly_{d,\underline{n}}(X)$ is equal to
	\begin{align*}
		&\phantom{=}~\det\left(\begin{array}{c|c}
			\begin{smallmatrix}
				X-1 & \phantom{-}X & -1 \\ 
				0 & -1 & \phantom{-}X \\ ~
			\end{smallmatrix} & \begin{smallmatrix}
				\phantom{-}0 &~ &~ &~ &  \scalebox{1.005}{$\cdots$} &~ & \phantom{-}0 \\ \phantom{-}0 &~&~ &~ &  \scalebox{1.005}{$\cdots$} &~ & \phantom{-}0 \\ ~
			\end{smallmatrix} \\ \hline
			\begin{smallmatrix}
				~ \\ ~\\
				1 & ~ &~ & 0 &~ & \phantom{-}1 \\ ~ \\1 &~ &~ & 0 &~ & \phantom{-}1 \\ \vdots &~ &~ & \vdots &~ & \phantom{-}\vdots\\ ~ \\ 1 &~ &~ & 0 &~ & \phantom{-}1
			\end{smallmatrix} & \begin{smallmatrix}
				~\\ ~ \\ ~ & 0\phantom{-} & \scalebox{1.00}{$\cdots$} & \phantom{-}0\\ ~ \\ ~ & \phantom{-}X^{n_3}+1 &~ & ~ \\ ~ & ~ &\ddots & ~ \\ ~\\~ & ~ & ~ & X^{n_{m}}+1
			\end{smallmatrix}
		\end{array}\right)+(X^{n_2}+1)\det\left(\begin{array}{c|c}
			\begin{smallmatrix}
				X-d & \phantom{-}0 & -(d-1) \\ 
				X-1 & \phantom{-}X & -1 \\ 
				0 & -1 & \phantom{-}X \\ ~
			\end{smallmatrix} & \begin{smallmatrix}
				-1 & ~ &~ &  \scalebox{1.005}{$\cdots$}~ & & -1 \\ \phantom{-}0 & ~&~ &  \scalebox{1.005}{$\cdots$} &~ & \phantom{-}0 \\ \phantom{-}0 &~  &~ &  \scalebox{1.005}{$\cdots$} &~ & \phantom{-}0 \\ ~
			\end{smallmatrix} \\ \hline
			\begin{smallmatrix}
				~ \\ ~ \\1 &~ &~ & 0 &~ &~ & \phantom{-}1 \\ \vdots &~ &~ & \vdots &~ &~ & \phantom{-}\vdots\\ ~ \\ 1 &~ &~ & 0 &~ &~ & \phantom{-}1
			\end{smallmatrix} & \begin{smallmatrix}
				~ \\ ~\\
				\phantom{-}X^{n_3}+1 &~ & ~ \\ ~  &\ddots & ~ \\ ~\\~ & ~ & X^{n_{m}}+1
			\end{smallmatrix}
		\end{array}\right) \\
		&= (X-1)X\prod_{i \in I_m \setminus \{2\}}(X^{n_i}+1)+(X^{n_2}+1)\det(A_2),
	\end{align*}
	where $A_2$ denotes the matrix obtained by deleting the fourth row and fourth column of the matrix in \eqref{eq: laplace}, which corresponds to the cofactor of the entry $X^{n_2}+1$. Denote by $A_j$ the matrix obtained by also deleting the fourth row and fourth column of $A_{j-1}$. Then, with induction, we find 
	\begin{align*}
		\det(A_j)&=(X-1)X\prod_{i >j}(X^{n_i}+1)+(X^{n_j}+1)\det(A_{j+1})\\ &=(X-1)X\sum_{k > j} \prod_{i \neq k}(X^{n_i}+1) + \prod_{i>j}(X^{n_i}+1)\det(A_{2d-1}).
	\end{align*} Note that \[\det(A_{2d-1})=\det\left(\begin{smallmatrix}
		X-d & \phantom{-}0 & -(d-1) \\ X-1 & \phantom{-}X & -1 \\ 0 &-1 & \phantom{-}X
	\end{smallmatrix}\right) = X^3-dX^2+(d-2)X+1=(X-1)(X^2-(d-1)X-1).\]
	Thus, $\charpoly_{d,\underline{n}}(X)=(X-1)\left((X^2-(d-1)X-1)\prod_{i=2}^{m}(X^{n_i}+1)+X\sum_{i=2}^{m} \prod_{j \in I_m \setminus \{i\}}(X^{n_j}+1)\right)=(X-1)p_{d, \underline{n}}(X)$, as claimed.
\end{proof}

Using this correspondence between these polynomials and the characteristic polynomial of the matrices $J_{d}^{\underline{n}}$ in \eqref{eq: block matrix}, one can prove that there is at most one root of $p_{d, \underline{n}}(X)$ whose absolute value is larger than $1$, and this root is real. To show this, we first prove the following lemma, for which we define $Q=\left(\begin{smallmatrix}
	1 & ~ & ~ & ~ \\ ~& -1 & ~ & ~ \\ ~& ~& \ddots & ~\\ ~& ~&~& -1
\end{smallmatrix}\right)$. Also, we write $A^T$ for the transpose of a matrix $A$.

\begin{lemma} \label{lemma: adapted intersection form}
	For $d\geq 1$, $1\leq m\leq 2d-1$ and $\underline{n}=(n_2, \ldots, n_m)$, the matrix $J_d^{\underline{n}}$ satisfies $J_d^{\underline{n}}Q(J_d^{\underline{n}})^T=Q+H_{d,m}$ where \[H_{d,m}=\left(\begin{array}{cc|ccc}
		\phantom{-}2d-1-m & -(2d-1-m) & 0 & \cdots & 0 \\
		-(2d-1-m) & \phantom{-}2d-1-m & 0 & \cdots & 0 \\
		\hline
		\phantom{-}0 & \phantom{-}0 & ~& ~& ~\\
		\phantom{-}\vdots & \phantom{-}\vdots & ~& \bf{0}& ~\\
		\phantom{-}0 & \phantom{-}0 & ~& ~& ~\\
	\end{array}\right).\]
	The matrix $H_{d,m}$ is positive semidefinite.
\end{lemma}
\begin{proof}
	The first part of the claim follows by a direct calculation using $Q$ and $J_{d}^{\underline{n}}$ given in \eqref{eq: block matrix}. As the characteristic polynomial of $H_{d,m}$ is equal to the product of $X-2(2d-1-m)$ and some power of $X$, and since $m\leq 2d-1$, all eigenvalues of $H_{d,m}$ are non-negative. Thus, $H_{d,m}$ is positive semidefinite.
\end{proof}

Equipped with this lemma, we can prove that there is at most one root of $p_{d,\underline{n}}(X)$ whose absolute value is greater than $1$, and this must be a real root. Note that the idea of the proof comes directly from the proof of Theorem~$5.1$, Assertion~($1$) of Diller and Favre~\cite{MR1867314}, yet we do not use the Theorem itself as we only deal with matrices and polynomials, and do not want to construct any surface on which they may or may not be realised as birational maps. 

\begin{proposition} \label{prop: root real and strictly larger than 1}
	Fix $d\geq 1$, $1\leq m\leq 2d-1$ and $\underline{n}=(n_2,\ldots, n_m)$. Then, $p_{d, \underline{n}}(X)$ has at most one root whose modulus is strictly larger than $1$, and this root is real.
\end{proposition}
\begin{proof}
	By Lemma~\ref{lemma: characteristic poly of all de Jonquieres}, $(X-1)p_{d,\underline{n}}(X)=\charpoly_{d,\underline{n}}(X)$ where $\charpoly_{d,\underline{n}}(X)$ is the characteristic polynomial of the matrix $J_d^{\underline{n}}$ defined in \eqref{eq: block matrix}. Thus, a root of $p_{d, \underline{n}}(X)$ corresponds to an eigenvalue of $J_d^{\underline{n}}$. Suppose by contradiction that we have two eigenvalues $\mu_1, \mu_2 \in \C$ of $J_d^{\underline{n}}$ with $|\mu_1|, |\mu_2|>1$, and denote by $v_1, v_2$ some corresponding eigenvectors. We claim that on the subspace $\C v_1+ \C v_2$, the bilinear form corresponding to $Q$ is positive semidefinite. This is equivalent to proving that $v_i^TQ\overline{v_i}\geq 0$ for $i=1,2$ and $(v_1^TQ\overline{v_1})(v_2^TQ
	\overline{v_2})-|v_1^TQ\overline{v_2}|^2 \geq 0$. 
	
	Abbreviate $H=H_{d,m}$ the matrix from Lemma~\ref{lemma: adapted intersection form}. Note that $\mu_i\overline{\mu_j}v_i^TQ\overline{v_j}=v_i^TJ_d^{\underline{n}}Q(J_d^{\underline{n}})^T\overline{v_j}\stackrel{\ref{lemma: adapted intersection form}}{=}v_i^TQ\overline{v_j}+v_i^TH\overline{v_j}$ for $i,j=1,2$, and thus $v_i^TQ\overline{v_i}=\tfrac{v_i^TH\overline{v_i}}{|\mu_i|^2-1}$ and  $v_1^TQ\overline{v_2}=\tfrac{v_1^TH\overline{v_2}}{\mu_1\overline{\mu_2}-1}$. Since $H$ is positive semidefinite by Lemma~\ref{lemma: adapted intersection form}, we directly see that $v_i^TQ\overline{v_i}\geq 0$ holds. Furthermore,
	\[(v_1^TQ\overline{v_1})(v_2^TQ
	\overline{v_2})-|v_1^TQ\overline{v_2}|^2=\tfrac{1}{(|\mu_1|^2-1)(|\mu_2|^2-1)}(v_1^TH\overline{v_1})(v_2^TH\overline{v_2})-\tfrac{1}{|\mu_1\overline{\mu_2}-1|^2}|v_1^TH\overline{v_2}|^2.\]
	Since $(|\mu_1|^2-1)(|\mu_2|^2-1)\leq|\mu_1\overline{\mu_2}-1|^2$ by expanding $|\mu_1-\mu_2|^2\geq 0$, we deduce
	\[(v_1^TQ\overline{v_1})(v_2^TQ\overline{v_2})-|v_1^TQ\overline{v_2}|^2\geq\tfrac{1}{(|\mu_1|^2-1)(|\mu_2|^2-1)}((v_1^TH\overline{v_1})(v_2^TH\overline{v_2})-|v_1^TH\overline{v_2}|^2)\geq0,\]
	where the last inequality is due to $H$ being positive semidefinite by Lemma~\ref{lemma: adapted intersection form}. Thus, $Q$ is positive semidefinite on $\C v_1 + \C v_2$. But the signature of $Q$ is $(1, \dim(Q)-1)$, so there cannot be a subspace of dimension $2$ on which $Q$ is positive semidefinite. Therefore, $v_1=v_2$ and $\mu_1=\mu_2$. As $\overline{\mu_1}$ is also an eigenvalue of $J_d^{\underline{n}}$, we additionally deduce $\mu_1=\overline{\mu_1}$, and thus that the root of largest absolute value must be real. This finishes the proof of the proposition.
\end{proof}

Thus, $p_{d,\underline{n}}(X)$ has a root of largest absolute value which is real, and we may therefore define:

\begin{definition} \label{def: largest real roots of polynomials}
	For $d\geq 1$, $1\leq m\leq 2d-1$ and $\underline{n}$ we denote by $\lambda_{d, \underline{n}} \in \R$ the root of largest absolute value of $p_{d, \underline{n}}(X)$.
\end{definition}

We are interested in describing $\lambda_{d,\underline{n}}(X)$; by the following proposition, it must be larger than $2$ under mild assumptions.

\begin{proposition} \label{prop: greater than 2}
	Fix $d\geq 4$, $1\leq m \leq 2d-1$ and $\underline{n}=(n_2, \ldots, n_m)$ with $n_i \geq 2$ for all $2\leq i \leq m$. Then, $\lambda_{d, \underline{n}}$ is strictly larger than $2$.
\end{proposition}
\begin{proof}
	The claim for $m=1$ follows by observing that in that case $\underline{n}$ is the empty tuple and thus $p_{d,\emptyset}(X)=X^2-(d-1)X-1$. Then, its largest root $\lambda$ satisfies $\lambda-\tfrac{1}{\lambda}=d-1$ and thus $2<d-1<\lambda$. Hence, assume $m\geq2$. The polynomial $p_{d,\underline{n}}(X)$ is monic. Thus, if at $X=2$, its value is negative, there must be a zero of $p_{d, \underline{n}}(X)$ which is larger than $2$. The value $p_{d,\underline{n}}(2)$ is strictly negative if and only if 
	\begin{align*}
		2\sum_{i=2}^{m}\prod_{j \in I_m\setminus \{i\}}(2^{n_j}+1)<(2d-5)\prod_{i=2}^{m}(2^{n_i}+1).
	\end{align*}
	Note that $4<2^{n_i}+1$ since $n_i\geq 2$. Therefore:
	\begin{align*}
		2\sum_{i=2}^{m}\prod_{j \in I_m\setminus \{i\}}(2^{n_j}+1)&<2\frac{1}{4}\sum_{i=2}^{m}\prod_{j =2}^{m}(2^{n_j}+1) = \frac{1}{2}(m-1)\prod_{i =2}^{m}(2^{n_i}+1) \\
		&\leq\frac{1}{2}(2d-2)\prod_{i =2}^{m}(2^{n_i}+1)\leq (2d-5)\prod_{i =2}^{m}(2^{n_i}+1),
	\end{align*}
	where in the last step we used $d\geq 4$. This proves the claim.
\end{proof}

Whenever we increase one of the $n_i$ defining some $p_{d,\underline{n}}(X)$, the largest root also increases strictly.

\begin{proposition} \label{prop: increase one parameter, root is larger}
	Fix $d\geq 4$, $2\leq m \leq 2d-1$ and $\underline{n}=(n_2, \ldots, n_m)$ with $n_i\geq 2$. For $2 \leq k \leq m$, denote $\underline{n}_k=(n_2, \ldots, n_{k-1}, n_k+1, n_{k+1}, \ldots, n_{m})$. Then $\lambda_{d, \underline{n}} < \lambda_{d,\underline{n}_k}$.
\end{proposition}
\begin{proof}
	The polynomials $p_{d,\underline{n}}(X)$ and $p_{d,\underline{n}_k}(X)$ are related by
	\[(X^{n_k}+1)p_{d,\underline{n}_k}(X)=(X^{n_k+1}+1)p_{d,\underline{n}}(X)-(X-1)X^{n_k+1}\prod_{j\in I_m\setminus \{k\}}(X^{n_j}+1).\]
	By Proposition~\ref{prop: greater than 2}, we know that $\lambda_{d,\underline{n}}>2$. Since $p_{d,\underline{n}}(\lambda_{d,\underline{n}})=0$, we deduce $p_{d,\underline{n}_k}(\lambda_{d, \underline{n}})<0$. The polynomial $p_{d, \underline{n}_k}(X)$ is monic, which implies that it must have a zero which is strictly larger than $\lambda_{d, \underline{n}}$, and thus $\lambda_{d,\underline{n}_k}>\lambda_{d, \underline{n}}$, which implies the claim.
\end{proof}

As a last property of the auxiliary polynomials, we prove that if we fix $n_2, \ldots, n_m \geq 2$ and consider all $\lambda_{d, \underline{n}'}$ where $\underline{n}'=(n_2, \ldots, n_m, n_{m+1})$ with $n_{m+1}\geq2$, then the sequence $(\lambda_{d, \underline{n}'})_{n_{m+1}}$ has $\lambda_{d, \underline{n}}$ as its limit.

\begin{lemma} \label{lemma: adding one parameter is smaller}
	Fix $d\geq 4$, $1\leq m < 2d-1$ and $\underline{n}=(n_2, \ldots, n_m)$, and $\underline{n}'=(n_2, \ldots, n_m, n_{m+1})$, with all $n_i\geq 2$. Then, $\lambda_{d,\underline{n}'}<\lambda_{d, \underline{n}}$. 
\end{lemma}
\begin{proof}
	The two polynomials $p_{d,\underline{n}}(X)$ and $p_{d,\underline{n}'}(X)$ are related as $p_{d, \underline{n}'}(X)=(X^{n_{m+1}}+1)p_{d,\underline{n}}(X)+X\prod_{j\in I_m}(X^{n_j}+1)$. By Proposition~\ref{prop: greater than 2}, $\lambda_{d, \underline{n}'}>2$, therefore, plugging in $\lambda_{d, \underline{n}'}$, and using that it is a root of $p_{d,\underline{n}'}(X)$, we find
	\[p_{d,\underline{n}}(\lambda_{d, \underline{n}'})=-\lambda_{d,\underline{n}'}\prod_{j\in I_m}(\lambda_{d, \underline{n}'}^{n_j}+1)<0\] 
	As $p_{d, \underline{n}}(X)$ is monic, this proves $\lambda_{d, \underline{n}'}<\lambda_{d, \underline{n}}$. 
\end{proof}

\begin{proposition} \label{prop: limit of polynomials}
	Fix $d\geq 4$, $2\leq m< 2d-1$ and $\underline{n}=(n_2, \ldots, n_{m-1})$ with $n_i\geq 2$. Denote $\underline{n}_{n_m}=(n_2, \ldots, n_{m-1}, n_m)$. Then the limit of the sequence $(\lambda_{d, \underline{n}_{n_m}})_{n_m}$ exists and is equal to $\lambda_{d, \underline{n}}$.
\end{proposition}
\begin{proof}
	Fix some $\underline{n}=(n_2, \ldots, n_{m-1})$ and $2 < \lambda < \lambda_{d, \underline{n}}$. Note that the polynomials satisfy \[p_{d, \underline{n}_{n_m}}(X)=(X^{n_m}+1)p_{d,\underline{n}}(X)+X\prod_{j \in I_{m-1}}(X^{n_j}+1).\] Note that since by Proposition~\ref{prop: root real and strictly larger than 1}, $\lambda_{d, \underline{n}}$ is the only real root of $p_{d,\underline{n}}(X)$ which is larger than $1$, on $(2,\lambda_{d, \underline{n}})$, the polynomial $p_{d, \underline{n}}(X)$ is negative. Thus, there exists $n_m(\lambda)\geq 2$ such that for all $n_m\geq n_m(\lambda)$, we have
	\[p_{d,\underline{n}_{n_m}}(\lambda)=(\lambda^{n_m}+1)p_{d,\underline{n}}(\lambda)+\lambda\prod_{j \in I_{m-1}}(\lambda^{n_j}+1)<0.\]
	This implies that for $n_m\geq n_m(\lambda)$, we have $\lambda<\lambda_{d, \underline{n}_{n_m}} < \lambda_{d, \underline{n}}$, where the last inequality is due to Lemma~\ref{lemma: adding one parameter is smaller}. This proves the claim.
\end{proof}

To finish this section, we prove that we can construct a well ordered set out of the $\lambda_{d, \underline{n}}$ such that we can precisely describe the ordinal of that set. 

\begin{proposition} \label{prop: ordinal of helping sets}
	Consider for $d\geq 4$ and $m \in \{2, \ldots, 2d-1\}$ the inductively defined sets
	\begin{align*}
		\Lambda_{d,1} &= \{ \tfrac{1}{2}(d-1+\sqrt{d^2-2d+5}) \}, \\
		\Lambda_{d,2}&=\{\lambda_{d, (n_2)} \thinspace | \thinspace  2\leq n_2, d-1 < \lambda_{d, (n_2)} \}, \\
		\Lambda_{d,m} &= \{\lambda_{d, (n_2, \ldots, n_m)} \thinspace | \thinspace 2\leq n_2 < \ldots < n_m, \lambda_{d, (n_2, \ldots, n_{m-1}-1)} \in \Lambda_{d,m-1}, \lambda_{d, (n_2, \ldots, n_{m-1}-1)} < \lambda_{d, (n_2, \ldots, n_m)} \}.
	\end{align*}
	The well ordering on the sets $\Lambda_{d,m}$ with respect to the lexicographic ordering on $(n_2, \ldots, n_m)$ agrees with the well ordering inherited from the one on the real line, and the ordinal of $\Lambda_{d, m}$ with respect to this ordering is~$\omega^{m-1}$. Furthermore, the sets $\Lambda_{d, m}$ are disjoint for all $d\geq 4$ and $m \in \{1, \ldots, 2d-1\}$.
\end{proposition}
\begin{proof}
	We prove the first two claims simultaneously by induction. For $m=1$ it is clear, and for $m=2$, the elements of $\Lambda_{d,2}$ indexed by $n_2$ form by Lemma~\ref{lemma: adding one parameter is smaller} an ascending chain of elements and thus $\Lambda_{d, 2}$ is well ordered with ordinal $\omega$. Now assume we know for some $m \in \{2,\ldots, 2d-2\}$ that the well ordering on $\Lambda_{d, m}$ given by the lexicographic ordering of the indexing set is the same as the one inherited from the real line and that $\Lambda_{d, m}$ has ordinal~$\omega^{m-1}$ with respect to this ordering. Define $J_m=\{(n_2, \ldots, n_m) \in \Z^{m-1} \thinspace | \thinspace \lambda_{d, (n_2, \ldots, n_{m}-1)}, \lambda_{d, (n_2, \ldots, n_m)} \in \Lambda_{d, m}\}$ and for each $(n_2, \ldots, n_m) \in J_m$, define the open interval $I_{(n_2, \ldots, n_m)} = (\lambda_{d, (n_2, \ldots, n_{m}-1)}, \lambda_{d, (n_2, \ldots, n_m)})$. For each $(n_2, \ldots, n_m) \in J_m$, there exists $n_{m+1}(\lambda_{d, (n_2, \ldots, n_{m}-1)}) > n_m$ such that for $n_{m+1} \geq n_{m+1}(\lambda_{d, (n_2, \ldots, n_{m}-1)})$, by Proposition~\ref{prop: limit of polynomials}, we know $\lambda_{d, (n_2, \ldots, n_{m+1})} \in I_{(n_2, \ldots, n_m)}$, and by Proposition~\ref{prop: increase one parameter, root is larger}, these values form a strictly increasing sequence with limit $\lambda_{d, (n_2, \ldots, n_m)}$. Thus, the sequence within $I_{(n_2, \ldots, n_m)}$ is well ordered, and as we can write $\Lambda_{d, m+1}$ as a union of such sequences, each contained in disjoint open intervals. Therefore, $\Lambda_{d, m+1}$ is well ordered in regard to the lexicographic ordering on its indexing set, too, and this ordering agrees again with the ordering on $\R$. Each sequence has ordinal~$\omega$, and by induction, $\Lambda_{d, m}$ has ordinal~$\omega^{m-1}$; this implies that $\Lambda_{d, m+1}$ has ordinal~$\omega^m$, as claimed. 
	
	As for the last claim, note that for a fixed $d\geq 4$, the $\Lambda_{d, m}$ with $m\in \{1, \ldots, 2d-1\}$ are disjoint by construction and by Lemma~\ref{lemma: adding one parameter is smaller}. For $4\leq d< d'$, note that all elements of $\Lambda_{d', m}$ for all $m \in \{1,\ldots, 2d'-1\}$ are larger than $d'-1$ by construction. Furthermore, all elements of $\Lambda_{d, m}$ with $m\in \{2,\ldots, 2d-1\}$ are smaller than $\tfrac{1}{2}(d-1+\sqrt{d^2-2d+5})$, which is smaller than $d\leq d'-1$ for $d\geq 4$. This proves the proposition.
\end{proof}

\begin{remark} \label{rem: Lambdas}
	Note that the sets $\Lambda_{d, m}$ do not contain any accumulation points in the open interval topology on $\R$. In fact, for any $m\in\{2,\ldots, 2d-1\}$, the accumulation points of the set $\Lambda_{d, m}$ lie in $\Lambda_{d, m-1}$ by Proposition~\ref{prop: limit of polynomials}, and the closure of $\Lambda_{d, m}$ is equal to $\bigsqcup_{1 \leq i \leq m}\Lambda_{d,i}$. 
\end{remark}

\section{Bounding from below} \label{section: bounding from below}

\subsection{Bounding from below using the Weyl group} \label{section: bounding from below, Weyl group}

We introduce the Weyl group in an brief way; for a more comprehensive introduction see~\cite{MR1066460}, \cite{MR1890629} or \cite{MR2354205}. The Weyl group can be defined as the Coxeter group $W_n=W(E_n)$ given by the Coxeter-Dynkin diagram $E_n$ on $n$ vertices as in Figure~\ref{figure: E_n}.

\begin{figure}[h]
	\begin{center}
		\begin{tikzpicture}
			\path 
			(0,1) node [style={draw,circle,minimum size=2mm,inner sep=0pt,outer sep=0pt,fill=black}, draw]{} node[above=1pt] {}
			(1,1) node [style={draw,circle,minimum size=2mm,inner sep=0pt,outer sep=0pt,fill=black}, draw]{} node[above=1pt] {}
			(2,1) node [style={draw,circle,minimum size=2mm,inner sep=0pt,outer sep=0pt,fill=black}, draw]{} node[above=1pt] {}
			(3,1) node [style={draw,circle,minimum size=2mm,inner sep=0pt,outer sep=0pt,fill=black}, draw]{} node[above=1pt] {}
			(4,1) node [style={draw,circle,minimum size=2mm,inner sep=0pt,outer sep=0pt,fill=black}, draw]{} node[above=1pt] {}
			(5,1) node [style={draw,circle,minimum size=2mm,inner sep=0pt,outer sep=0pt,fill=black}, draw] {} node[above=1pt] {}
			(2,0) node [style={draw,circle,minimum size=2mm,inner sep=0pt,outer sep=0pt,fill=black}, draw] {} node[below=1pt] {};
			\draw (0,1) -- (1,1) -- (2,1) -- (3,1);
			\draw [dashed] (3,1) -- (4,1);
			\draw (4,1) -- (5,1);
			\draw (2,1) -- (2,0);
		\end{tikzpicture}
		\caption{The graph $E_n$.} \label{figure: E_n}
	\end{center}
\end{figure}
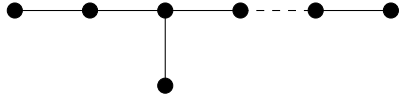

Alternatively, fix $n\geq 3$. Consider $\Z^{1,n}$ the lattice $\Z^{n+1}$ with inner product $x \cdot x =x_0^2-x_1^2-\cdots -x_n^2$ and standard basis $(e_0,e_1,\ldots, e_n)$. Define the \emph{Weyl group} $W_n$ as the subgroup of $O(\Z^{1,n})$ generated by the reflections $x \mapsto x+(x\cdot \alpha_i) \alpha_i$, where $\alpha_0=e_0-e_1-e_2-e_3$ and $\alpha_i=e_i-e_{i+1}$ for $i=1, \ldots, n-1$. There exists an isometry 
\begin{align*}
	\Z^{1,n} &\rightarrow \Pic(X), \quad
	e_0 \mapsto L, \quad
	e_i \mapsto E_i, \quad i=1, \ldots, n
\end{align*}
for a rational surface $X$ obtained as the blow-up of $\PP^2$ in $n$ points out of which we can pick four points where no three are collinear, and where $L$ is the pullback of a general line and $E_i$ are the exceptional curves. Via this isometry, $W_n$ is the subgroup of $\Aut(\Pic(X)) \cong \GL_{n+1}(\Z)$ generated by the permutation matrices with the upper left entry fixed as $1$, and by the matrix 
\begin{align*} A_0&=\left(\begin{array}{c|c}
		\begin{smallmatrix}
			2& \phantom{-}1& \phantom{-}1& \phantom{-}1\\ -1& \phantom{-}0 & -1 & -1\\ -1 & -1 & \phantom{-}0 & -1 \\ -1 & -1 & -1 & \phantom{-}0 \\ ~
		\end{smallmatrix} &   \\
		\hline  & \mathbbm{1}
	\end{array} \right).
\end{align*} 
Even more, Nagata~\cite{MR126444} proved that the image of $\Aut(X)$ under the homomorphism $\Aut(X)\rightarrow \Aut(\Pic(X))$ lies in $W_n$. Note that $W_n$ must not necessarily be equal to the subgroup consisting of elements preserving the intersection form and the anticanonical divisor; take for example the reflection determined by $3e_0 - \sum_{i=1}^{10}e_i + e_{11}$.

Define the \emph{Weyl spectrum} $\Lambda(W)$ as the set of all spectral radii of all elements of all $W_n$ with $n\geq 3$. By the following result of Uehara, rephrased to fit the framework of this article, the Weyl spectrum agrees with all possible dynamical degrees of rational surface automorphisms.

\begin{theorem}[{\cite{MR3477879}, Theorem~$1.1$}] \label{thm: uehara}
	If the ground field is $\C$, then the Weyl spectrum $\Lambda(W)$ agrees with \[\Lambda(\{f \in \Bir(\PP^2) \enspace | \enspace \exists \enspace \pi \colon X \rightarrow \PP^2\colon \pi^{-1}f\pi \in \Aut(X) \}).\]
\end{theorem}

With Uehara's result, we can infer the lower bound~$\omega^\omega$ on the ordinal of $\Lambda(\Bir(\PP^2))$ by showing it for $\Lambda(W)$.

\begin{lemma} \label{lemma: ordinal Weyl spectrum}
	The Weyl spectrum $\Lambda(W)$ is a well ordered subset of $\R$ and its ordinal is greater or equal to~$\omega^\omega$.
\end{lemma}
\begin{proof}
	The Weyl spectrum $\Lambda(W)$ is a subset of $\Lambda(\Bir(\PP^2)) \subset \R$, which is well ordered by \cite[{Theorem~$7.2$}]{MR3454379}, thus so is~$\Lambda(W)$. We claim that all the matrices $J_d^{\underline{n}}$ in~\eqref{eq: block matrix} with $m=2d-1$ belong to the respective Weyl group $W_n$ with $n=\sum_{i=2}^{2d-1}n_i+2$. Multiplying with suitable permutations $P_1$ and $P_1'$ lying in $W_n$ and $A_0$, we can bring $J_d^{\underline{n}}$ into the following form:
	\begin{align*}
		P_1'J_d^{\underline{n}}P_1A_0=\left(\begin{array}{c|c}
			\begin{smallmatrix}
				\phantom{-}d & \phantom{-}d-1 & \phantom{-}1 & \phantom{-}1\\ -(d-1)&-(d-2) & -1 & -1 \\-1 & -1 & -1 & \phantom{-}0 \\ -1 & -1 & \phantom{-}0 & -1 \\ ~
			\end{smallmatrix} & \ast \\ \hline \ast & \ast
		\end{array}\right)\left(\begin{array}{c|c}
			\begin{smallmatrix}
				\phantom{-}2& \phantom{-}1& \phantom{-}1& \phantom{-}1\\ -1& \phantom{-}0 & -1 & -1\\ -1 & -1 & \phantom{-}0 & -1 \\ -1 & -1 & -1 & \phantom{-}0 \\ ~
			\end{smallmatrix} &   \\
			\hline  & \mathbbm{1}
		\end{array} \right)=\left(\begin{array}{c|c}
			\begin{smallmatrix}
				\phantom{-}d-1 & \phantom{-}d-2 & \phantom{-}0 & \phantom{-}0\\ -(d-2)&-(d-3) & \phantom{-}0 & \phantom{-}0 \\ \phantom{-}0 & \phantom{-}0 & \phantom{-}0 & \phantom{-}1 \\ \phantom{-}0 & \phantom{-}0 & \phantom{-}1 & \phantom{-}0 \\ ~
			\end{smallmatrix} & \ast \\ \hline \ast & \ast
		\end{array}\right).
	\end{align*}
	As there are $2d-2$ columns of the form $P(1,-1,-1,0, \ldots, 0)^T$ with $P$ a permutation leaving the first entry fixed, we can repeat this process $(d-1)$-times and obtain that $P_{d-1}'\cdots P_1'J_d^{\underline{n}}P_1A_0 \cdots P_{d-1}A_0$ is a permutation matrix lying in $W_n$. Thus, $J_d^{\underline{n}}$ is indeed an element of $W_n$. 
	
	Therefore, by Lemma~\ref{lemma: characteristic poly of all de Jonquieres}, the disjoint sets $\Lambda_{d, 2d-1}$ defined in Proposition~\ref{prop: ordinal of helping sets} lie in $\Lambda(W)$ and thus $\bigsqcup_{d \geq 4} \Lambda_{d, 2d-1} \subset \Lambda(W)$. By Proposition~\ref{prop: ordinal of helping sets}, each $\Lambda_{d, 2d-1}$ for $d\geq 4$ is of ordinal $\omega^{2d-2}$, and thus their disjoint union is of ordinal~$\omega^\omega$. We deduce that the ordinal of $\Lambda(W)$ is also at least~$\omega^\omega$, which proves the lemma.
\end{proof}

\begin{theorem}\label{thm: main theorem, but with weyl group}
	If the base field is $\C$, then the ordinal of $\Lambda(\Bir(\PP^2))$ is greater or equal to $\omega^\omega$, and the ordinal of the Weyl spectrum $\Lambda(W)$ is equal to $\omega^\omega$.
\end{theorem}
\begin{proof}
	Since by Theorem~\ref{thm: uehara}, we have the inclusion $\Lambda(W) \subset \Lambda(\Bir(\PP^2))$, and thus by Lemma~\ref{lemma: ordinal Weyl spectrum} the ordinal of $\Lambda(\Bir(\PP^2))$ is also greater or equal to $\omega^\omega$. This proves the first claim. The second claim follows by applying Theorem~\ref{thm: bounding from above}.
\end{proof}

Proposition~\ref{thm: main theorem, but with weyl group} together with Theorem~\ref{thm: bounding from above} imply Theorem~\ref{thm: main theorem}, and the second claim of Theorem~\ref{thm: main theorem, but with weyl group} is precisely Theorem~\ref{thm: main theorem for Weyl spectrum}.

\subsection{Bounding from below using explicit realisations} \label{section: bounding from below, realisations}

To prove that the ordinal of $\Lambda(\Bir(\PP^2))$ can be bounded from below by $\omega^\omega$ without appealing to Uehara's result in~\cite{MR3477879}, we prove the claim that for any $d \geq 4$, the ordinal of $\Lambda(\Bir_d(\PP^2))$ is strictly larger than $\omega^{2d-2}$. Knowing the latter, we can deduce the former by taking the union over all $\Lambda(\Bir_d(\PP^2))$. To show this claim, we construct suitable sequences of birational maps having dynamical degrees that depend on $2d-2$ independent parameters, and where the dynamical degree increases strictly whenever we increase one of the parameters. The birational maps we choose are \emph{de Jonqui\`eres maps}, which send the pencil of lines through a given point to the pencil of lines through another fixed point. We first show that their dynamical degrees indeed have the desired properties, and then prove that the maps also exist, in a sense working backwards. The existence will be easier to establish by choosing the base points to lie on an irreducible cuspidal curve. Note that in this section, for the constructions to work, the field of definition needs to contain the field of real algebraic integers; this implies further that $\mathbf{k}$ is of characteristic $0$.

\subsubsection{\textbf{Arithmetic properties of the de Jonqui\`eres matrix}} \label{subsection: arithmetic}

We first introduce some general notation and results, where now we can borrow heavily from Section~\ref{section: arithmetic properties of auxiliary polynomials}. Consider a birational map $f \in \Bir(\PP^2)$. Write $\Base(f)=\{p_1, \ldots, p_m\}$ for the base points of $f$ and $\Base(f^{-1})=\{q_1, \ldots, q_m\}$ for the base points of its inverse. Assume that they are all proper, which by definition means that none of the base points are infinitely near to one another, thus no tangent directions being fixed at any of the points. Suppose that for $1 \leq i \leq m$, we have $n_i \in \Z^{\geq 0}$ such that $f^{n_i-1}(q_i)=p_i$ and assume that the points $f^s(q_t)$ for $1 \leq t \leq m$ and $0 \leq s \leq n_t-1$ are pairwise distinct. Call the tuple $(n_1, \ldots, n_m)$ the \emph{orbit data of $f$}, the set $O(q_t)=\{f^s(q_t) \thinspace | \thinspace 0 \leq s \leq n_t-1 \}$ the \emph{orbit of $q_t$} and $O(f)=\bigcup_{1 \leq t \leq m}O(q_t)$ the \emph{orbit of $f$}. 

Note that if we blow up the full orbit $O(f)$, then $f$ lifts to an automorphism on the blow-up. This automorphism induces an automorphism of the Picard group via pushforward, i.e. $f_\ast:[D] \mapsto [f(D)]$, which can be represented by an integer-valued matrix after fixing a basis $L, \{E(p)\}_{p \in O(f)}$. Of this matrix, which we will denote by $F$, we know that the spectral radius $\rho(F)$ is equal to $\lambda(f)$; in addition, since $f$ preserves the canonical divisor, one eigenvalue of $F$ must be $1$. Also, call $L$ the pullback of a general line in $\PP^2$ not passing through any of the points of the full orbit, and any exceptional curve $E(p)$ above a point $p \in O(f)$. Setting
\begin{align} Q=\left(\begin{smallmatrix}
		1 & & & \\
		& -1 & & \\
		& & \ddots & \\
		& & & -1
	\end{smallmatrix} \right),\label{eq: Q}\tag{$\clubsuit$}
\end{align} 
the fact that $f$ preserves the intersection form translates to $FQF^T=Q$, where $Q$ is the matrix given in \eqref{eq: Q} and $F^T$ denotes the transpose of~$F$.

We would now like to find such birational maps whose dynamical degrees are strictly greater than $1$ and increase in a strict way whenever we increase one of the $n_i$, for which we consider the matrices we obtain from de Jonqui\`eres maps. Recall that a \emph{de Jonqui\`eres map} of degree~$d$ is a map $j_d \in \Bir(\PP^2)$ with base point $p_1$ of multiplicity~$d-1$ and base points $p_2, \ldots, p_{2d-1}$ of multiplicity~$1$. Call the base points of the inverse $q_1, \ldots, q_{2d-1}$, where $q_1$ is the unique one of multiplicity~$d-1$ and all others have multiplicity~$1$. Write $j_d^{\underline{n}}$ for a de Jonqui\`eres map with orbit data $\underline{n}=(n_1, \ldots, n_{2d-1})$ and $\lambda_{d,\underline{n}}$ for its dynamical degree. Also, since it will make computations easier, we set $n_1=2$ and abbreviate $\underline{n}=(n_2, \ldots, n_{2d-1})$. Yet for our purposes, this restriction will not impact what we set out to do, which is to make a statement on the ordinal. Thus, these birational maps will depend on $2d-2$ parameters. 

A priori, we do not know if these de Jonqui\`eres maps $j_d^{\underline{n}}$ can be realised; we will first show arithmetically that their corresponding matrices satisfy all the desired properties, and then prove that under suitable circumstances, they are indeed realised as birational maps of $\PP^2$. Calling the -- so far hypothetical -- points $q_{i}^0=q_i$ and $q_{i}^{k}=(j_d^{\underline{n}})^k(q_i)$ for $1\leq k \leq n_i-1$, and their exceptional curves $E(q_{i}^k)$ in some blow-up $X$, we order the basis of $\Pic(X)$ in which to write $J_d^{\underline{n}}$ as $L, E(q_1^0)=E(q_1), E(q_1^1)=E(p_1), E(q_2^0)=E(q_2), \ldots, E(q_2^{n_2-1})=E(p_2),  \ldots, E(q_{2d-1}^{n_{2d-1}-1})=E(p_{2d-1})$, where $L$ is a general line. In this basis, the matrix corresponding to $j_d^{\underline{n}}$ takes precisely the for of $J_d^{\underline{n}}$ defined in~\eqref{eq: block matrix} for $m=2d-1$. We call such a $J_d^{\underline{n}}$ a \emph{de Jonqui\`eres matrix.} We may thus summarise the arithmetic properties of the de Jonqui\`eres matrices gleaned from Section~\ref{section: arithmetic properties of auxiliary polynomials}:

\begin{proposition} \label{prop: properties of jonquieres}
	The de Jonqui\`eres matrices with $d\geq 4$ and $n_2, \ldots, n_{2d-1} \geq 2$ have the following properties:
	\begin{enumerate}[leftmargin=*]
		\item The spectral radius of $J_d^{\underline{n}}$ is a real eigenvalue of $J_d^{\underline{n}}$ and strictly greater than $2$. \label{item: properties of jonquieres, 1}
		\item Increasing an orbit $n_i$ by one increases the spectral radius of the corresponding de Jonqui\`eres matrix strictly. \label{item: properties of jonquieres, 2}
	\end{enumerate}
\end{proposition}
\begin{proof}
	Claim~\ref{item: properties of jonquieres, 1} follows from Proposition~\ref{prop: greater than 2}, and Claim~\ref{item: properties of jonquieres, 2} follows from Proposition~\ref{prop: increase one parameter, root is larger}.
\end{proof}

\subsubsection{\textbf{Existence of de Jonqui\`eres maps}} \label{subsection: existence of jonquieres maps}

Now that we know the de Jonqui\`eres maps to be suitable candidates, we want to see that they are realised as actual birational maps of $\PP^2$. For this, we consider the singular cuspidal cubic $C = V(xy^2-z^3)$ in $\mathbb{P}^2$ with unique singular point $[1:0:0]$, and set $C^\ast =C \setminus \{[1:0:0]\}$ the set of smooth points. The isomorphism $\A^1 \overset{\sim}{\rightarrow} C^\ast$, $t \mapsto [t^3:1:t]$ allows to parametrise the smooth points of $C$ with $\A^1$ and endow $C^\ast$ with a group structure, for which $t_1+t_2+t_3=0$ if and only if the points $t_1, t_2, t_3$ on $C^\ast$ are collinear. We can say even more, and even though the following lemma is well-known, we reproduce it for completeness' sake:

\begin{lemma} \label{lemma: on curve of degree d, sum is zero}
	If for $3d \geq 3$ not necessarily distinct points $t_i \in C^\ast$, there exists a polynomial $P$ of degree $d$ whose restriction to $C$ is not zero and such that for $p_i=[t_i^3:1:t_i]$ we have $P|_C=p_1+\cdots+p_{3d}$, then $\sum_{i=1}^{3d}t_i=0$.
\end{lemma}
\begin{proof}
	Suppose we have $3d$ points $p_i=[t_i^3 : 1:t_i]$ with $t_i\in C^\ast$ such that $P|_C=p_1+\cdots+p_{3d}$ for a polynomial $P\in \C[X,Y,Z]$ of degree $d$ not a multiple of the defining equation of $C$. As $P(t^3,1,t)\in C[t]$ is a polynomial of degree at most $3d$ and $P|_C=p_1+\cdots+p_{3d}$, we must have $P(t^3,1,t)=a\prod_{i=1}^{3d}(t-t_i)$ for some $a\in \C^\ast$. The coefficient of $t^{3d-1}$ has to be zero, as $P$ is homogeneous and we replaced $X,Y,Z$ by $t^3,1,t$, respectively. But this coefficient is precisely equal to $-a\sum_{i=1}^{3d}t_i=0$, which finishes the proof.
\end{proof}

For a fixed orbit data $\underline{n}$, we will choose distinct points lying on the singular cuspidal cubic $C$. Then, after blowing up these points, the automorphism group of the blow-up will lie within the Weyl group, of which the de Jonqui\`eres matrix $J_d^{\underline{n}}$ is a part of. Lastly, we will prove that $J^{\underline{n}}_d$ gives rise to an automorphism of the blow-up, thus implying the existence of a de Jonqui\`eres map $j^{\underline{n}}_d$.
This line of analysis has already been pursued by McMullen~\cite{MR2354205} or Uehara~\cite{MR3477879}, and we will recall the main tools for it.

\begin{lemma} \label{lemma: kernel and eigenvector, arithmetic}
	Consider any matrix $F \in \GL_{n+1}(\Z)$ satisfying $FQF^T=Q$, where $Q$ is defined as in \eqref{eq: Q}, and $F(3,1, \ldots, 1)^T=(3,1,\ldots, 1)^T$. Fix $a \in \C\setminus \{1\}$ and $b,v_1,\ldots, v_n \in \C$. If the vector $(b,v_1, \ldots, v_n)^T$ is an eigenvector of $F^T$ with respect to $a$, then we have the following system of equations:
	\begin{align}
		3b&=F_{21}\frac{1}{a-1}(3v_1-b)+\cdots +F_{n+1,1}\frac{1}{a-1}(3v_n-b), \label{eq: kernel and eigenvector, 1}\tag{$\heartsuit$} \\
		b&=F_{2,i+1}\frac{1}{a-1}(3v_1-b)+ \cdots + (F_{i+1,i+1}-a)\frac{1}{a-1}(3v_i-b)+\cdots +F_{n+1,i+1}\frac{1}{a-1}(3v_n-b). \label{eq: kernel and eigenvector, 2}\tag{$\spadesuit$}
	\end{align}
\end{lemma}
\begin{proof}
	Assume that the vector $(b,v_1, \ldots, v_n)^T$ is an eigenvector of $F^T$ with respect to the eigenvalue $a$. Then $(b,v_1, \ldots, v_n)^T$ satisfies
	\begin{align*}
		(F_{11}-a)b+F_{21}v_1+\cdots +F_{n+1,1}v_n&=0, \\
		F_{1,i+1}b+F_{2,i+1}v_1+ \cdots + (F_{i+1,i+1}-a)v_i+\cdots +F_{n+1,i+1}v_n&=0, \quad 1 \leq i \leq n.
	\end{align*}
	By multiplying each equation by $3$, and subtracting and adding $\sum_{j=2}^{n+1}F_{j1}b$ from the first equation and $\sum_{j=2}^{n+1}F_{j,i+1}b$ from the remaining $n$ equations, the equations imply
	\begin{align*}
		3(F_{11}-a)b+F_{21}(3v_1-b)+\cdots +F_{n+1,1}(3v_n-b)+\sum_{j=2}^{n+1}F_{j1}b&=0, \\
		(3F_{1,i+1}+\sum_{j=2}^{n+1}F_{j,i+1}-a)b+F_{2,i+1}(3v_1-b)+ \cdots + (F_{i+1,i+1}-a) (3v_i-b)+\cdots +F_{n+1,i+1}(3v_n-b)&=0, 
	\end{align*}
	where $1\leq i \leq n$. By assumption, since $(3,1,\ldots, 1)F=(3,1,\ldots, 1)$, we have $3F_{11}+\sum_{j=2}^{n+1}F_{j1}=3$ and $3F_{1,i+1}+\sum_{j =2}^{n+1}F_{j,i+1}=1$, and thus find
	\begin{align*}
		3(1-a)b+F_{21}(3v_1-b)+\cdots +F_{n+1,1}(3v_n-b)&=0, \\
		(1-a)b+F_{2,i+1}(3v_1-b)+ \cdots + (F_{i+1,i+1}-a) (3v_i-b)+\cdots +F_{n+1,i+1}(3v_n-b)&=0, 
	\end{align*}
	which after dividing by $(a-1)$ implies \eqref{eq: kernel and eigenvector, 1} and \eqref{eq: kernel and eigenvector, 2}.
\end{proof}

Using this arithmetic property of Lemma~\ref{lemma: kernel and eigenvector, arithmetic}, we can prove the following geometric statement.

\begin{proposition} \label{prop: eigenvector implies existence on cuspidal cubic}
	Let $C \subset \PP^2$ be a singular cuspidal cubic given by $C=V(xy^2-z^3)$, fix $n\geq 10$ and let $p_1, \ldots, p_n \in C^\ast$ be distinct points. Call $\pi \colon X \rightarrow \PP^2$ the blow-up of $\PP^2$ in the points $p_1, \ldots, p_n$. Denote by $E(p_i)=\pi^{-1}(p_i)$ the exceptional curves. Consider $F \in \Aut(\Pic(X)) \setminus \{\id\}$, and call its corresponding matrix $F$, too. Suppose that $F$ sends the anticanonical divisor $-K_X$ to itself and that it preserves the intersection form. Furthermore, assume that for each $i\in \{1,\ldots, n\}$ there exists an irreducible curve $C_i$ of $X$ in the class of $F(E(p_i))$. Fix an eigenvalue $a \in \C\setminus \{1\}$ of $F$ and consider an eigenvector $(b,v_1, \ldots, v_n)^T$ of $F^T$. If $p_i=\frac{1}{a-1}(3v_i-b)$, then there exists an automorphism $f \in \Aut(X)$ such that the induced action on $\Pic(X)$ is equal to $F$, and the restriction of $f$ to the strict transform $\widetilde{C}$ of $C$ induces an automorphism of $C^\ast\cong \A^1$ given by $z \mapsto az +b$.
\end{proposition}
\begin{proof}
	By assumption, there are irreducible curves $C_i$ linearly equivalent to $F(E(p_i))$. As $E(p_i)\cdot E(p_j)=-\delta_{ij}$ and $F$ preserves the intersection form, we thus find $C_i\cdot C_j=-\delta_{ij}$; hence we obtain a contraction $\eta \colon X \rightarrow \PP^2$ of the curves $C_i$. Note that since we blow up more than $10$ points, the strict transform $\widetilde{C}$ of $C$ is the only irreducible curve linearly equivalent to $-K_X$. Thus, since $E(p_i)\cdot (-K_X)=1$, any $E(p_i)$ meets $\widetilde{C}$ transversally in one point. Again since $F$ preserves the intersection form and since $F(-K_X)=-K_X$ by assumption, any $C_i$ meets $\widetilde{C}$ in one point transversally. As $\widetilde{C}\sim -K_X$, the birational morphism $\eta$ maps $\widetilde{C}$ to an irreducible cubic curve $\eta(\widetilde{C})$. Additionally, the $\eta(C_i)$ cannot be singular points of $\eta(C)$, since $\widetilde{C}\cdot C_i=1$. Since none of the points blown up equals the singular point and since an irreducible cubic curve can have at most one double point, $\eta(C)$ is again cuspidal. We can therefore choose an automorphism $\beta \in \Aut(\PP^2)$ which maps $\eta(\widetilde{C})$ to $C$, and by replacing $\eta$ with $\beta \circ \eta$ we may assume that $\eta(\widetilde{C})=C$.
	
	As we blow up at least ten points on $C^\ast$, the strict transform $\widetilde{C}$ of $C$ will be the only irreducible curve lying in the anticanonical divisor class, and since $F$ sends $-K_X$ to itself, it must equally send $\widetilde{C}$ to itself. Therefore, we see that $\pi \circ \eta^{-1}$ restricts to an automorphism of $C$, say given by $z \mapsto a'z+b'$. By assumption, there exists an eigenvalue $a\neq1$ of $F$. Any automorphism of $C$ given by $z \mapsto \hat{a}z$, where $\hat{a} \in C^\ast \cong \A^1$, extends to some automorphism $\gamma_{\hat{a}} \in \Aut(\PP^2)$. Hence, by replacing $\eta$ by $\gamma_{aa'^{-1}}^{-1}\circ \eta$, we may assume that $\pi \circ \eta^{-1}$ descends to the automorphism $C^\ast \rightarrow C^\ast, z \mapsto az+ b'$. Denote $\varphi_F=\pi \circ \eta^{-1}$.
	
	We first prove $b'=b$. For this, consider the pullback $\eta^{-1}(L_{\PP^2})$ of a line $L_{\PP^2}$ not passing through any of the $\eta(C_i)$ via $\eta$, and not being tangent to $C$. We know $\eta^{-1}(L_{\PP^2})\sim F(L)$ for a general line $L$ on $X$. Call $r_1, r_2, r_3$ the three points of intersection of $L_{\PP^2}$ and $C$. As $L_{\PP^2}$ and $C$ intersect in three points, the curve $\eta^{-1}(L_{\PP^2})$ intersects $\widetilde{C}$ in three points.
	The images of $r_1, r_2, r_3$ under $\varphi_F|_{C^\ast}$ equal the images of these three points $\eta^{-1}(r_1), \eta^{-1}(r_2), \eta^{-1}(r_3)$ under $\pi$. Since $F(L)\sim F_{11}L+F_{21}E_1+\cdots + F_{n+1,1}E_{n}$, the curve $\pi(\eta^{-1}(L_{\PP^2}))$ is a curve of degree $F_{11}$ passing through $p_i$ with multiplicity~$-F_{i+1,1}$ and through $\varphi_F|_{C^\ast}(r_1), \varphi_F|_{C^\ast}(r_2), \varphi_F|_{C^\ast}(r_3)$ with multiplicity~$1$ each. Thus, by the above, on $C^\ast$ we have the following equality:
	\[3b'\stackrel{\ref{lemma: on curve of degree d, sum is zero}}{=}a(r_1+r_2+r_3)+3b'=\varphi_F|_{C^\ast}(r_1)+\varphi_F|_{C^\ast}(r_2)+\varphi_F|_{C^\ast}(r_3)\stackrel{\ref{lemma: on curve of degree d, sum is zero}}{=}F_{21}p_1+\cdots + F_{n+1,1}p_n.\]
	By Lemma~\ref{lemma: kernel and eigenvector, arithmetic}~\eqref{eq: kernel and eigenvector, 1}, we know that $F_{21}p_1+\cdots + F_{n+1,1}p_n=3b$, and thus find $b'=b$.
	
	We analyse how $\varphi_F|_{C^\ast}$ maps the points $\eta(C_i)$. The curve $C_i$ must be linearly equivalent to $F(E(p_i))=F_{1,i+1}L+\sum_{j=1}^{n}F_{j+1,i+1}E_j$; this corresponds to the $(i+1)$-th column of $F$. As $F$ preserves the intersection form, we find $ C_i\cdot\widetilde{C}= C_i\cdot (-K_X)= F(E(p_i))\cdot F(-K_X)=E(p_i)\cdot (-K_X)=1$. Therefore, there is exactly one point lying in the intersection of $C_i$ and $\widetilde{C}$, say $\widetilde{p_i}$, and $\pi(\widetilde{p_i})=\varphi_F|_{C^\ast}(\eta(C_i))$. The point $\pi(\widetilde{p_i})$ lies on $\pi(C_i)=\pi(F(E(p_i)))$ and the curve $\pi(C_i)$ is a curve of degree $F_{1,i+1}$ with multiplicity~$-F_{j+1,i+1}$ at $p_j$ and multiplicity~$1$ at $\pi(\widetilde{p_i})$. Considering the divisor on $C$ given by $\pi(C_i)$, we thus find
	\[a\eta(C_i)+b=\varphi_F|_{C^\ast}(p_i)=\pi(\widetilde{p_i})\stackrel{\ref{lemma: on curve of degree d, sum is zero}}{=}F_{2,i+1}p_1 + \cdots + F_{n+1,i+1}p_n.\]
	By Lemma~\ref{lemma: kernel and eigenvector, arithmetic}~\eqref{eq: kernel and eigenvector, 2}, this must be equal to $ap_i+b$, which implies $\eta(C_i)=p_i$. Thus, $\varphi_F$ lifts to an automorphism on $X$, whose induced action on $\Pic(X)$ is equal to $F$ and whose restriction to $\widetilde{C}$ is the same automorphism as $\varphi_F|_{C^\ast}:z\mapsto az +b$.
\end{proof}

\begin{remark}
	A similar result holds for an arbitrary irreducible cubic, but one has to take a bit more care, see for instance Diller~\cite{MR2825269} for the degree~$2$ case. However, only very few dynamical degrees outside the unit circle arise on a non-cuspidal cubic (Diller~\cite[Theorem~$2$]{MR2825269}).
\end{remark}

With the above Proposition~\ref{prop: eigenvector implies existence on cuspidal cubic}, we have an ansatz for which points on $C^\ast$ we are looking for, and thus, we calculate the eigenvector of $(J^{\underline{n}}_d)^T$ with respect to the eigenvalue $\lambda_{d, \underline{n}}$ realising the spectral radius in the following lemma:

\begin{lemma} \label{lemma: eigenvector of transpose}
	Fix $\underline{n}$ with $n_1=2$ and $n_i \geq2$ for all $i \in \{2, \ldots, 2d-1\}$. Abbreviate $\lambda=\lambda_{d, \underline{n}}$. Then the vector $(\lambda+1, 1, \lambda, \frac{\lambda}{\lambda^{n_2}+1},\frac{\lambda^2}{\lambda^{n_2}+1}, \ldots, \frac{\lambda^{n_2}}{\lambda^{n_2}+1}, \ldots, \frac{\lambda^{n_{2d-1}}}{\lambda^{n_{2d-1}}+1})^T$ is an eigenvector of $(J^{\underline{n}}_d)^T$ with respect to the eigenvalue $\lambda$. 
\end{lemma}
\begin{proof}
	Since $\lambda$ is strictly larger than $2$ by Proposition~\ref{prop: greater than 2}, the vector is well-defined. We check directly that $(\lambda+1, 1, \lambda, \frac{\lambda}{\lambda^{n_2}+1},\frac{\lambda^2}{\lambda^{n_2}+1}, \ldots, \frac{\lambda^{n_2}}{\lambda^{n_2}+1}, \ldots, \frac{\lambda^{n_{2d-1}}}{\lambda^{n_{2d-1}}+1})^T$ is an eigenvector of $(J^{\underline{n}}_d)^T$ with respect to the eigenvalue $\lambda$. Indeed, for all the entries where $(J_d^{\underline{n}})^T$ only permutes the entries of the eigenvector, this follows by construction of the vector. The only equations we need to check are
	\begin{align*}
		d(\lambda+1)-(d-1)-\sum_{i=2}^{2d-1}\frac{\lambda}{\lambda^{n_i}+1} &= \lambda(\lambda+1), \\
		(d-1)(\lambda+1)-(d-2)-\sum_{i=2}^{2d-1}\frac{\lambda}{\lambda^{n_i}+1} &= \lambda^2, \\
		\lambda+1 -1-\frac{\lambda}{\lambda^{n_i}+1}&=\lambda\frac{\lambda^{n_i}}{\lambda^{n_i}+1}, \quad i\in \{2,\ldots, 2d-1\}.
	\end{align*}
	Multiplying the two upper equations by $\prod_{i =2}^{2d-1}(\lambda^{n_i}+1)$, we find the expression of the characteristic polynomial, which evaluates to $0$ at $\lambda$, and the last equation is also valid, which can be seen after multiplying both sides by $\lambda^{n_i}+1$.
\end{proof}

With Proposition~\ref{prop: eigenvector implies existence on cuspidal cubic} and Lemma~\ref{lemma: eigenvector of transpose}, we now have precise instructions on how to choose the points on $C^\ast$. The following proposition shows that with said choice on the points and a weak assumption on the orbits, we can ensure that $j_d^{\underline{n}}$ exists. Parts of the proof follow the same arguments as in the proof of Proposition~$5.12$ of Blanc and Cantat~\cite{MR3454379}.

\begin{proposition} \label{prop: jonquieres maps exist}
	Let $d \geq 4$ and denote by $C\subset \PP^2$ the cuspidal cubic $C=V(xy^2-z^3)$. Fix the orbit data $\underline{n}=(n_2,\ldots, n_{2d-1})$ with $n_i\geq 2$ for each $i\in \{2, \ldots, 2d-1\}$ and all the $n_i$ pairwise distinct. Denote by $\lambda \in \R$ the eigenvalue $\lambda_{d, \underline{n}}$ realising the spectral radius of the de Jonqui\`eres matrix $J_d^{\underline{n}}$ and set 
		\begin{align*}
		q_1^0=q_1&=\frac{1}{\lambda-1}(2-\lambda), \\ q_1^1=p_1&=\frac{1}{\lambda-1}(2\lambda-1), \\ q_i^0=q_i&=\frac{1}{\lambda-1}(3\frac{\lambda}{\lambda^{n_i}+1}-(\lambda+1)) , \\ q_i^j&=\frac{1}{\lambda-1}(3\frac{\lambda^{j+1}}{\lambda^{n_i}+1} -(\lambda+1))
		, \quad 2 \leq i \leq 2d-1, 1 \leq j \leq n_i-1,
	\end{align*}
	where the points are identified with their corresponding point on $C^\ast \cong \A^1$. Then there exists a de Jonqui\`eres map $j^{\underline{n}}_d \in \Bir_d(\PP^2)$ with base points $p_1, p_2=q_2^{n_2-1}, \ldots, p_{2d-1}=q_{2d-1}^{n_{2d-1}-1}$, orbits $O(q_i)=\{q_i^j \thinspace | \thinspace 0 \leq j \leq n_i-1\}$ and dynamical degree~$\lambda$.
\end{proposition}
\begin{proof}
	To prove the proposition, we show that there exists an automorphism realising the de Jonqui\`eres matrix $J_d^{\underline{n}}$ of \eqref{eq: block matrix} on the blow-up $\pi\colon X \rightarrow \PP^2$ in the points $q_i^j$ by applying Proposition~\ref{prop: eigenvector implies existence on cuspidal cubic}. By direct calculation, we see that $J_d^{\underline{n}}$ sends $-K_X$ to itself and preserves the intersection form. Thus, to apply Proposition~\ref{prop: eigenvector implies existence on cuspidal cubic}, we need to prove the following two claims: that the points $q_i^j$ are all distinct, and that in each class of $J_d^{\underline{n}}(E(q_i^j))$ there exists an irreducible curve, where $E(q_i^j)=\pi^{-1}(q_i^j)$ denotes the exceptional curve above $q_i^j$. 
	
	We start with proving that the points are distinct. Assume $q_1=p_1$. This is equivalent to $\lambda = 1$, a contradiction to $\lambda>2$ by Proposition~\ref{prop: properties of jonquieres}. Now, take any $i\neq 1$ and assume $q_1=q_i^j$. This is equivalent to $\lambda^{n_i}+1=\lambda^{j+1}$, a contradiction to $j\leq n_i-1$ and $\lambda >2$. Similarly, we find that $p_1=q_i^j$ is equivalent to $\lambda^{n_i}+1=\lambda^{j}$, again a contradiction. Then, for a fixed $i \neq 1$, we can also immediately see that $q_i^{k} \neq q_i^{\ell}$ whenever $k \neq \ell$. Assume therefore that $i,j \neq 1$ and that we have $0\leq k \leq n_i-1$ and $0 \leq \ell \leq n_j-1$ such that $q_i^k=q_j^\ell$. Assume without loss of generality that $n_i>n_j$. This is equivalent to $\lambda^k(\lambda^{n_j}+1)=\lambda^\ell(\lambda^{n_i}+1)$.
	Since $n_i>n_j$ and $\lambda>2$, the equation yields $k>\ell$ and is  equivalent to $\lambda^m(\lambda^{n_j}+1)=\lambda^{n_i}+1$, where $m=k-\ell$. As $\lambda>2$ it must follow that $\lambda^{m+n_j}<\lambda^m(\lambda^{n_j}+1)-1= \lambda^{n_i}$, and thus $m+n_j<n_i$. 
	Hence $\lambda^m(\lambda^{n_j}+1)<2\lambda^{m+n_j}<\lambda^{m+n_j+1}\leq \lambda^{n_i}<\lambda^{n_i}+1$, a contradiction. Therefore, none of the points blown up by $\pi$ agree.
	
	Now that we know that the point $q_i^j$ are distinct, we turn to proving that there exist irreducible curves in $X$ linearly equivalent to $J_d^{\underline{n}}(E(q_i^j))$. Note that using \eqref{eq: block matrix}, the only $J_d^{\underline{n}}(E(q_i^j))$ which are not equal to an exceptional divisor are:
	\begin{align*}
		J_d^{\underline{n}}(E(q_1^1))=J_d^{\underline{n}}(E(p_1))&= (d-1)L-(d-2)E(q_1)-\sum_{i =2}^{2d-1}E(q_i), \\
		J_d^{\underline{n}}(E(q_i^{n_i-1}))=J_d^{\underline{n}}(E(p_i))&= L-E(q_1)-E(q_i), \quad i \in \{2, \ldots, 2d-1\}, \\
		J_d^{\underline{n}}(E(q_i^j))&= E(q_i^{j+1}).
	\end{align*}
	We want to secure the existence of an irreducible degree~$d-1$ curve through $q_1$ with multiplicity~$d-2$ and through $q_2, \ldots, q_{2d-1}$ with multiplicity~$1$ each. 
	
	Start with the blow-up $\tau\colon X_1\rightarrow\PP^2$ of $\PP^2$ in $q_1$ to the first Hirzebruch surface $X_1=\F_1$. Blow up the point on $X_1$ corresponding to $q_2$ and blow down the strict transform of the line through $q_1$ and $q_2$, resulting in a birational map $\tau_2\colon X_1\dashrightarrow X_2$, where $X_2$ is either $\F_0$ or $\F_2$. By proving that for all $2\leq i<j\leq 2d-1$, the points $q_i$ and $q_j$ cannot lie on a line through $q_1$, we can repeat this procedure. In this way, we construct a sequence of birational maps $X_1 \overset{\tau_2}{\dashrightarrow}X_2 \overset{\tau_3}{\dashrightarrow} \cdots  \overset{\tau_{2d-1}}{\dashrightarrow} X_{2d-1}$, where at each step, $\tau_i$ equals the blow-up the point corresponding to $q_i$ followed by the blow-down of the curve we obtain from pushing forward the strict transform of the line through $q_1$ and $q_i$ by the map $\tau_{i-1}\cdots\tau_2$; moreover, $X_i$ is equal to some Hirzebruch surface $\F_s$ for some $s\geq 0$ with $s\equiv i \mod 2$. Thus assume by contradiction that there exist $2\leq i<j\leq 2d-1$ such that the points $q_1$, $q_i$ and $q_j$ lie on a line. By Lemma~\ref{lemma: on curve of degree d, sum is zero}, this implies $q_1+q_i+q_j=0$, which is equivalent to $\lambda^{n_i+n_j}=1$, a contradiction to $\lambda > 2$. Therefore, a sequence of birational maps $\tau_{2d-1}\cdots \tau_2$ as described above does indeed exist.
	
	Now, $X_{2d-1}$ is equal to some $\F_s$ with $s$ odd. We prove $s=1$. Assume there exists an irreducible curve of self-intersection~$\leq -3$. It must come from a curve in $\PP^2$ of degree~$k$ passing through $q_1$ with multiplicity~$k-1$ and through $\ell$ points out of $q_2,\ldots, q_{2d-1}$ with multiplicity~$1$. The self-intersection of its strict transform in $X_1$ is $k^2-(k-1)^2=2k-1$. Whenever one of the $\ell$~points through which it passes is blown up, the self-intersection decreases by one; yet for every point it does not pass through, the self-intersection increases by one, since we contract a $(-1)$-curve passing through that point. Thus, the irreducible curve of self-intersection~$\leq -3$ has self-intersection $2k-1-\ell+2d-2-\ell=2(d+k-\ell)-3\leq -3$. Therefore, $d+k\leq \ell$. But, as all the $q_i$ lie on the cubic curve $C$, with B\'ezout, we find for any curve of $\PP^2$ of degree~$k$ passing through $q_1$ with multiplicity~$k-1$ and $\ell$ other points with multiplicity~$1$ that $3k-(k-1)-\ell \geq 0$, implying $2k+1\geq \ell \geq d+k$ and thus $k\geq d-1$ and $\ell\geq 2d-1$. But $\ell \leq 2d-2$, a contradiction. Hence, $X_{2d-1}=\F_1$, and we can contract the $(-1)$-curve to obtain a birational transformation $f\colon \PP^2\dashrightarrow\PP^2$ as in the following diagram, where $\pi$ decomposes into the blow-up $\pi_1$ of the points $q_1,\ldots, q_{2d-1}$, and $\pi_2$ the blow-up of the remaining $q_i^j$: 
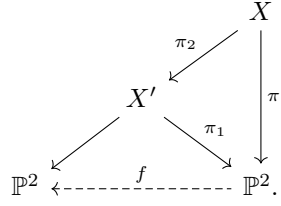
	\begin{figure}[H]
			\centering
			\begin{tikzcd}
						~ & ~& X  \arrow[dd, "\pi"] \arrow[dl, "\pi_2"']  \\
						~ & X' \arrow[dl] \arrow[dr, "\pi_1"] & ~  \\
						\PP^2 \arrow[rr, dashleftarrow, "f"] & ~ & \PP^2.
			\end{tikzcd}
	\end{figure} 
	The irreducible curve of self-intersection $-1$ on $X_{2d-1}$ must by the same reasoning as above come from an irreducible curve on $\PP^2$ of degree~$d-1$ passing through $q_1$ with multiplicity~$d-2$ and through all the other $q_2, \ldots, q_{2d-1}$ with multiplicity~$1$, which lifts to an irreducible curve on $X'$ linearly equivalent to $(d-1)L-(d-2)E(q_1)-\sum_{i=2}^{2d-1}E(q_i)$. Furthermore, any line through $q_1$ and $q_i$ lifts to an irreducible curve in $X'$.
	
	To verify that the lines through $q_1$ and $q_i$ and the curve of degree~$d-1$ described above indeed lift to $X$ as desired, we need to prove that none of the other points $q_i^j$, $j>0$, which are blown up by $\pi_2$, lie on one of these lines through $q_1$ and $q_k$, or on the curve of degree $d-1$ passing through $q_1$ with multiplicity~$d-2$ and through $q_2, \ldots, q_{2d-1}$ with multiplicity~$1$. 
	
	We prove that none of the $q_i^j$ with $j>0$ lie on a line through $q_1$ and $q_k$: if that were the case, then by Lemma~\ref{lemma: on curve of degree d, sum is zero}, we would find $q_1+q_k+q_i^j=0$. Consider first $q_1+q_k+p_1=0$. This is equivalent to $\lambda=0$, a contradiction. If now $i>1$, we find that $q_1+q_k+q_i^j=0$ holds if and only if $\tfrac{\lambda}{\lambda^{n_k}+1}+\tfrac{\lambda^{j+1}}{\lambda^{n_i}+1}=1$. Both terms on the left can be bounded from above strictly by $\tfrac{1}{2}$ as long as $j+1<n_i-1$, providing a contradiction. If $j+1=n_i-1$, then we find $(\lambda^{n_i}+1)\lambda=\lambda^{n_k}+1$. Therefore, $\lambda^{n_k}+1=\lambda^{n_i+1}+\lambda>\lambda^{n_i+1}+1$, whence $n_k>n_i+1$. Then, $(\lambda^{n_i}+1)\lambda<2\lambda^{n_i+1}<\lambda^{n_i+2}\leq \lambda^{n_k}<\lambda^{n_k}+1$, again a contradiction. Hence, no other $q_i^j$ lies on a line through $q_1$ and $q_k$.

	We need to check that none of the orbits lie on the degree~$d-1$ curve through $q_1, \ldots, q_{2d-1}$ with all points having multiplicity~$1$ except $q_1$ with multiplicity~$d-2$. By Lemma~\ref{lemma: on curve of degree d, sum is zero}, this implies $q_i^j=-(d-2)q_1-\sum_{k =2}^{2d-1}q_k$. By replacing the $q_i$ with their explicit values given in the statement and using $\lambda\sum_{k =2}^{2d-1}\prod_{\ell \neq k}(\lambda^{n_\ell}+1)=-(\lambda^2-(d-1)\lambda-1)\prod_{k =2}^{2d-1}(\lambda^{n_k}+1)$, we find
	\begin{align*}
		(\lambda -1)((d-2)q_1+\sum_{k =2}^{2d-1}q_k)&=(d-2)(2-\lambda)+\sum_{k =2}^{2d-1}(3\frac{\lambda}{\lambda^{n_k}+1}-(\lambda+1)) \\
		&=(d-2)(2-\lambda)+\frac{3}{\prod_{k =2}^{2d-1}(\lambda^{n_k}+1)}\lambda\sum_{k =2}^{2d-1}\prod_{\ell \neq k}(\lambda^{n_\ell}+1)-(2d-2)(\lambda+1) \\
		&=(d-2)(2-\lambda)-3(\lambda^2-(d-1)\lambda-1)-(2d-2)(\lambda+1) \\
		&= -3\lambda^2+(\lambda+1).
	\end{align*}
	Therefore, $q_i^j=-(d-2)q_1-\sum_{k =2}^{2d-1}q_k$ would imply 
	$\lambda^2+1=0$ if $q_i^j=q_1$, $\lambda^2+\lambda=0$ if $q_i^j=p_1$ and $\lambda^2(\lambda^{n_i}+1)+\lambda^j=0$ otherwise.
	But for any of these equations, the left side is positive, since $\lambda>2$, leading to the desired contradiction.

	By the above arguments, there exist irreducible curves $J_d^{\underline{n}}(E(q_i^j))$ on the blow-up $\pi \colon X\rightarrow \PP^2$ in the points $q_i^j$. Thus, applying Proposition~\ref{prop: eigenvector implies existence on cuspidal cubic}, there exists a de Jonqui\`eres map $j_d^{\underline{n}} \in \Bir_d(\PP^2)$ with base points $p_1,\ldots, p_{2d-1} \in C^\ast$ and orbits $O(q_i)=\{q_i^j\thinspace |\thinspace 0\leq j\leq n_i-1\}$ which lifts to an automorphism on $X$ such that the induced action on $\Pic(X)$ is equal to the de Jonqui\`eres matrix $J_d^{\underline{n}}$. As the spectral radius of $J_d^{\underline{n}}$ is equal to $\lambda$ by assumption, $j_d^{\underline{n}}$ has dynamical degree~$\lambda$. This finishes the proof.
\end{proof}	

Knowing that the de Jonqui\`eres maps which are based on the de Jonqui\`eres matrices truly exist, we can prove the main statement of this section:

\begin{theorem} \label{thm: bounding from below}
	If $d\geq 4$, then the ordinal of $\Lambda(\Bir_d(\PP^2))$ is greater or equal to $\omega^{2d-2}$. Thus, the ordinal of $\Lambda(\Bir(\PP^2))$ is greater or equal to $\omega^\omega$.
\end{theorem}
\begin{proof}
	By Proposition~\ref{prop: ordinal of helping sets}, if we set $m=2d-1$, we find that $\Lambda_{d, 2d-1} \subset \Lambda(\Bir_d(\PP^2))$ is of order type $\omega^{2d-2}$, and it is contained in the open interval $(d-1, d)$. Then, considering $\bigsqcup_{d\geq 4} \Lambda_{d, 2d-1} \subset \Lambda(\Bir(\PP^2))$, we find that the order type of $\Lambda(\Bir(\PP^2))$ is at least $\omega^\omega$. 
\end{proof}

Combining Theorem~\ref{thm: bounding from above} and Theorem~\ref{thm: bounding from below} implies Theorem~\ref{thm: main theorem}.

\begin{remark} \label{rem: last remark}
	As noted in Remark~\ref{rem: Lambdas}, the accumulation points of $\Lambda_{d, 2d-1}$ lie in $\Lambda_{d, 2d-2}$, which are not Salem numbers any more, but Pisot numbers. \emph{Salem numbers }are algebraic integers strictly larger than $1$ whose other Galois conjugates lie in the closure of the unit disk, with at least one of the conjugates on the boundary, and a \emph{Pisot number} is an algebraic integer strictly larger than $1$ whose other Galois conjugates lie in the open unit disk (see for example Blanc and Cantat~\cite[$1.1.2$]{MR3454379}). One can see that the elements of $\Lambda_{d, 2d-1}$ are Salem and the elements of $\Lambda_{d, m}$ with $m<2d-1$ are Pisot by noting that $m=2d-1$ if and only if $H_{d,m}$ from Lemma~\ref{lemma: adapted intersection form} is the zero matrix.
	
	The main difference between Salem and Pisot numbers is that any birational map which can be realised as an automorphism on some blow-up must be Salem, and any birational map which cannot is Pisot, except for quadratic reciprocal integers, which can occur in both cases. Yet by Blanc and Cantat~\cite[Theorem~D]{MR3454379}, the set $\Lambda(\Bir(\PP^2))$ is closed as soon as $\mathbf{k}$ is uncountable and algebraically closed (for example $\mathbf{k}=\C$), and thus the accumulation points $\Lambda_{d, 2d-2}$ and in fact all elements of the closure $\overline{\Lambda_{d, 2d-1}}=\bigsqcup_{1 \leq m \leq 2d-1}\Lambda_{d, m}$ are also realised as dynamical degrees of birational maps. 
\end{remark}

\small{

\bibliographystyle{alpha}
\bibliography{refs}

}
\noindent

\bigskip\noindent
Anna Bot ({\tt annakatharina.bot@unibas.ch}),\\
Department of Mathematics and Computer Science, University of Basel, 4051 Basel, Switzerland

\end{document}